\documentclass[article,11pt]{amsart}
\usepackage[top=25mm,bottom=25mm,left=25mm,right=25mm]{geometry}

\newtheorem{lemma}{Lemma}[section]

\newtheorem{theorem}{Theorem}[section]
\newtheorem{corollary}{Corollary}[section]
\newtheorem{definition}{Definition}[section]
\newtheorem{example}{Example}[section]
\newtheorem{remark}{Remark}[section]

\usepackage{color}

\makeatletter
\def\section{\@startsection{section}{1}%
\z@{1\linespacing\@plus\linespacing}{1\linespacing}%
{\bf\centering}}
\def\subsection{\@startsection{subsection}{0}%
\z@{\linespacing\@plus\linespacing}{\linespacing}%
{\bf}}
\makeatother

\makeatletter
\@addtoreset{equation}{section}
\makeatother

\DeclareMathOperator{\supp}{supp}
\DeclareMathOperator{\Dom}{Dom}
\DeclareMathOperator{\Spec}{Spec}

\DeclareMathOperator{\dist}{dist}

\DeclareMathOperator{\loc}{loc}

\DeclareMathOperator{\slim}{s-lim}

\providecommand{\pro}[1]{(#1_t)_{t \geq 0}}

\newcommand{\cK}{\mathcal{K}}

\newcommand{\cB}{\mathcal{B}}

\newcommand{\cE}{\mathcal{E}}

\newcommand{\R}{\mathbf{R}}
\newcommand{\1}{\mathbf{1}}

\newcommand{\pr}{\mathbf{P}}

\newcommand{\ex}{\mathbf{E}}

\newcommand{\Rd}{\mathbf{R}^d}
\newcommand{\N}{\mathbf{N}}

\begin{document}
\title[Zero-energy bound state decay for non-local Schr\"odinger operators]
{Zero-energy bound state decay for non-local Schr\"odinger operators}
\author{Kamil Kaleta and J\'ozsef L\H orinczi}
\address{Kamil Kaleta, Faculty of Pure and Applied Mathematics \\
 Wroc{\l}aw University of Science and Technology
\\ Wyb. Wyspia{\'n}skiego 27, 50-370 Wroc{\l}aw, Poland}
\email{kamil.kaleta@pwr.edu.pl}

\address{J\'ozsef L\H orinczi,
Department of Mathematical Sciences, Loughborough University \\
Loughborough LE11 3TU, United Kingdom}
\email{J.Lorinczi@lboro.ac.uk}

\thanks{
\noindent
2010 {\it MS Classification}: Primary 47D08, 60G51; Secondary 47D03, 47G20 \\
\noindent
KK was supported by the National Science Center (Poland) grant 2015/18/E/ST1/00239 and by the Alexander von Humboldt
Foundation (Germany). JL thanks IHES, Bures-sur-Yvette, where part of this paper has been written.
}

\begin{abstract}
We consider solutions of the eigenvalue equation at zero energy for a class of non-local Schr\"odinger operators
with potentials decreasing to zero at infinity. Using a path integral approach, we obtain detailed results on the
spatial decay of both $L^2$ and resonance solutions at infinity. We highlight the interplay of the kinetic term
and the potential in these decay behaviours, and identify the decay mechanisms resulting from specific balances
of global lifetimes with or without the potential.
\end{abstract}

\maketitle

\baselineskip 0.5 cm

%\medskip
\section{Introduction}
\noindent
The study of spectral properties of Schr\"odinger operators $H=-\frac{1}{2}\Delta + V$ on $L^2(\R^d)$, featuring the
Laplacian and a potential $V$, has a long history in mathematics. Potentials decaying to zero at infinity produce a
fascinating variety of spectral behaviours and phenomenology, including the possibility of finite or countably infinite
discrete spectra, (dense sets of) embedded eigenvalues or singular continuous spectrum, resonances, criticality, Efimov
effect, enhanced binding, or scattering. In this, there is a split of qualitative behaviours according to the rates of
decay of the potential (which led to various concepts of `long range' and `short range' potentials). The results indicate
that the existence of embedded eigenvalues is a long range effect, and the appearance of positive point spectrum is a
combination of slow decay and oscillations of the potential. For surveys we refer to \cite{RS3,EK,CK01,DS07} and the
numerous references therein.

Zero-energy level, which coincides with the edge of the continuous spectrum, often marks a borderline between various
regimes of spectral behaviour, in particular, between existence and non-existence of bound states, thus also shedding
light on the mechanisms of ``birth" of such states. Whether zero is an eigenvalue is, in general, a difficult problem.
Some early results on existence or non-existence of zero-energy eigenvalues go back to the papers
\cite{A70,N77,K78,JK79,L81,S81,R87,K89}. For potentials which are negative at infinity, decaying at a rate $V(x)\asymp -c
|x|^{-\gamma}$, $c > 0$, as $|x|\to\infty$, and satisfying some further conditions, it has been established in \cite{FS04}
that for $\gamma \in (0,2)$ zero is not an eigenvalue, see also \cite{DS09,SW}. For potentials that are positive at infinity
the situation changes \cite{Y82,N94}. In \cite{BY90} it has been shown that for Schr\"odinger operators on $L^2(\R^3)$
with rotationally symmetric potentials $V \in L^p(\R^3)$, $p > \frac{3}{2}$, whose positive part satisfies $V^+(x) \leq
C|x|^{-2}$ for $x$ large enough, zero is not an eigenvalue corresponding to a positive eigenfunction if $C =\frac{3}{4}$,
while a positive $L^2$-eigenfunction does exist if $C > \frac{3}{4}$. (The result holds more generally for non-symmetric
potentials and higher dimensions as well.) For some further results on the existence of compactly supported eigenfunctions
at zero-eigenvalue for compactly supported $V \in L^p(\R^d)$, $p < \frac{d}{2}$, we refer to \cite{KN00,KT02}. Apart from
some cases using direct methods of analysis, the two most used techniques leading to these results are based on unique
continuation or resolvent expansions.

Recently, a theory of non-local Schr\"odinger operators started to shape up, which has enriched the range of spectral
phenomena and opened a new perspective to the understanding of classical Schr\"odinger operators as a specific case.
Such operators arise by replacing the Laplacian with a suitable pseudo-differential operator such as the fractional
Laplacian $(-\Delta)^{\alpha/2}$, $0 < \alpha < 2$. In the present paper our primary aim is an analysis of eigenfunction
decay at zero-eigenvalue or zero-resonance for a class of non-local Schr\"odinger operators $H = -L + V$ on $L^2(\R^d)$,
with decaying potentials. One of our goals is to highlight the role of the potential in an interplay with the kinetic
operator term $-L$ in generating the decay behaviours.

The occurrence of zero or strictly positive eigenvalues for non-local Schr\"odinger operators just begins to be studied.
In the recent paper \cite{LS17}, see Theorems 2.8 and 2.10, two sets of potentials generating zero-eigenvalues or
zero-resonances for massless relativistic Schr\"odinger operators in dimension one have been constructed (as well as
examples leading to a strictly positive eigenvalue). More recently, in \cite{JL18} this has been generalized to fractional
Laplacians of all order and arbitrary dimensions. Let $\kappa>0$, $\alpha \in (0,2)$, and $P$ be a harmonic polynomial,
homogeneous of degree $l \geq 0$, i.e., satisfying $P(cx) = c^l P(x)$ for all $c>0$, and $\Delta P = 0$. Denote $\mu =
d+2l$, and consider the potentials and functions
\begin{equation}
\begin{aligned} \label{eq:motiv_ex}
V_{\kappa,\alpha}(x) & =  -\frac{2^\alpha}{\Gamma(\kappa)} \Gamma\left(  \frac{\mu+\alpha}{2}\right)
\Gamma\left( \frac{\alpha}{2} + \kappa \right)(1+|x|^2)^\kappa \,_2\textbf{F}_1\left( \left. \begin{array}{c}
				\frac{\mu+\alpha}{2} \quad \frac{\alpha}{2} + \kappa \\ \frac{\mu}{2}
			\end{array}  \right| -|x|^2 \right) \\
\varphi_\kappa(x) & = \frac{P(x)}{(1+|x|^2)^\kappa},
\end{aligned}
\end{equation}
where $_2\textbf{F}_1$ is Gauss' hypergeometric function. Then
\begin{equation*}
(-\Delta)^\frac{\alpha}{2} \varphi_\kappa + V_{\kappa,\alpha} \varphi_\kappa = 0
\end{equation*}
holds in distributional sense with $\varphi_\kappa \in L^2(\R^d)$ if $\kappa \geq \frac{\mu}{4}$, and
\begin{equation}
\label{exdecays}
|V_{\kappa,\alpha}(x)| =
\left\{
\begin{array}{lcl}
O\left( |x|^{-\alpha} \right) &\mbox{if}& \kappa \in (l, \frac{\mu}{2}) \setminus \{ \frac{\mu-\alpha}{2} \} \\
\vspace{0.1cm}
O\left( |x|^{-2\alpha} \right) &\mbox{if}& \kappa = \frac{\mu-\alpha}{2} \\
\vspace{0.1cm}
O\left( |x|^{-\alpha}\log|x| \right) &\mbox{if}& \kappa = \frac{\mu}{2} \\
\vspace{0.1cm}
O\left( |x|^{2\kappa-\mu-\alpha} \right) &\mbox{if}& \kappa \in (\frac{\mu}{2}, \frac{\mu+\alpha}{2}).
\end{array}
\right.
\end{equation}
Furthermore, for large $|x|$ we have that
%\vspace{0.1cm}
%\begin{enumerate}
%\item[(1)]
%$V_{\kappa,\alpha}(x) < 0$ if $\kappa \in \left( l , \frac{\mu-\alpha}{2} \right]$
%\medskip
%\item[(2)]
%$V_{\kappa,\alpha}(x) > 0$ if $\kappa \in \left( \frac{\mu-\alpha}{2}, \frac{\mu+\alpha}{2}  \right)$.
%\end{enumerate}
\begin{eqnarray}
&& \hspace{-4.4cm} V_{\kappa,\alpha}(x) < 0 \quad \mbox{if} \quad \kappa \in \big( l , \frac{\mu-\alpha}{2} \big]
\label{neg} \\
&& \hspace{-4.4cm} V_{\kappa,\alpha}(x) > 0 \quad \mbox{if} \quad \kappa \in \big(\frac{\mu-\alpha}{2},\frac{\mu+\alpha}{2}\big).
\label{pos}
\end{eqnarray}

%\medskip
We note that the above examples by no means indicate that zero eigenvalues are common or easy to locate. By using
methods of operator analysis, we have established further cases or conditions of existence as well as non-existence
of embedded eigenvalues in \cite{LS}. For $d=3$ it is known that provided $|V|$, $|x\cdot\nabla V|$ and $|x\cdot\nabla
(x\cdot\nabla V)|$ are bounded by $C(1+x^2)^{-1/2}$, with a small $C > 0$, jointly imply that $(-\Delta)^{1/2}+V$ has
no non-negative eigenvalue \cite{RU12}. Related work on unique continuation for fractional Schr\"odinger equations
imply further non-existence results \cite{FF14,S14,S15,R15}. Some further recent work include non-positive potentials
with compact support and $L$ chosen to be the massive relativistic operator \cite{M06}, and a class of generalized
Schr\"odinger operators \cite{C17}.

In our previous works \cite{KL15,KL17} we have investigated eigenfunction decay for a large class of non-local
Schr\"odinger operators. Using a probabilistic approach, $L$ was assumed to be the infinitesimal generator of a
`jump-paring' L\'evy process, i.e., having the property that
$$
\int_{|x-y|>1, \, |y|>1} \nu(|x-y|)\nu(|y|)dy \leq C \nu(|x|), \quad |x| > 1,
$$
where $\nu$ is the L\'evy jump density entering the symbol of the kinetic part $-L$ of the operator (for details see
the next section), and $C > 0$ is a constant. This condition means that double (and by iteration, any multiple) large
jumps are stochastically dominated by single large jumps. As we have shown, there is a large family of such operators
and related random processes of interest, including the fractional Laplacian and isotropic stable processes,
relativistic Laplace operators and relativistic stable processes, and many others.

In the first paper quoted above we considered confining potentials $V$, i.e., increasing to infinity as $|x| \to \infty$,
and found that, for instance, if the potentials grow at a sufficiently regular rate in the sense that $\sup_B V \leq c
\inf_B V$ for all unit balls $B$ far enough from the origin, and where $c > 0$ is a constant, then the ground state
(eigenfunction corresponding to the bottom of the spectrum) $\varphi_0$ behaves like $\varphi_0 \asymp \frac{\nu}{V}$. In
this case the contributions of the kinetic and potential terms in $H$ separate neatly. In the second paper we considered
decaying potentials, i.e., decreasing to zero as $|x| \to \infty$, such that $\lambda_0 = \inf \Spec H < 0$. In this case
both $V$ and $\nu$ decrease to zero, and now the interplay between the two terms determining the decay is much more
intricate. We have shown that the decay behaviour depends on how some quantities related to intrinsic jump preferences of
the L\'evy processes compare with the width of the gap between the ground state eigenvalue and the continuum edge.
Specifically, if $\lambda_0$ is sufficiently low-lying, then $\varphi_0 \asymp \nu$. Moreover, when the gap is small and
$\nu$ is chosen to have an increasingly light tail proceeding from a polynomial or sub-exponential, through an exponential,
to a super-exponential decay rate, a sharp regime change in the decay of $\varphi_0$ can be observed and the rate of decay
of the ground state suddenly becomes slower than the rate of decay of $\nu$, see also \cite{KL16a}.

The framework we have developed allows also to gain insight into the mechanisms generating these decay behaviours.
When $L$ is chosen such that the process $\pro X$ it generates has the jump-paring property and $V$ is an increasing
potential, we obtain
\begin{equation}
\label{meanntime}
\varphi_0(x) \, \asymp \, \Lambda_{B(x,1)} \nu(x), \quad \mbox{with} \quad \Lambda_{B(x,1)} =
\ex^x\left[\int_0^{\tau_{B(x,1)}} e^{-\int_0^t V(X_s)ds} dt \right],
\end{equation}
where $\tau_{B(x,1)} = \inf\{t> 0: \, X_t \not\in B(x,1)\}$ denotes the first exit time of the process from a unit ball
centered in $x$. This means that the fall-off rate depends on how soon the process perturbed by the potential on average
leaves unit balls far out. Dependent on how negative $\lambda_0$ is and how light the tails of $\nu$ are, this relationship
is preserved for decaying potentials up to a point when too light tails cause a drop below a critical level in the
domination of single large jumps, which the ground state energy cannot compensate. Then the sojourn times due to multiple
re-entries in unit balls of the process become comparable with the exit times piling up `backlogs' in the fall-off events,
and this slows the decay of $\varphi_0$ down. Indeed, when $\nu$ has so light tails that the jump-pairing condition no
longer holds, $\varphi_0$ decays necessarily (and possibly much) slower than $\nu$.

Our concern in the present paper is how $\varphi_0$ decays in the conditions when $\lambda_0 = \inf \Spec H = 0$. We note
that in classical results on ground state decay of usual (local) Schr\"odinger operators by Agmon, Carmona, and other
authors (see a discussion, e.g., in \cite[Ch. 3]{LHB}), a gap between the lowest eigenvalue and the edge of the continuous
spectrum is an essential ingredient, and the results break down when this gap is brought to zero.
% and the problem becomes indeed, this difference usually gives the rate of exponential decay. very delicate.
In \cite{ABG} it has been shown that for $d \geq 2$ a zero-energy eigenfunction $\varphi$ for a potential $V$ satisfying
$\big\||x|^{2- d/p}V\big\|_{L^p(\R^d)} < \infty$, for some $p \geq 1$ possibly being infinite, implies a power-law lower
bound on the decay given by $|x|^{-a}\varphi \not\in L^2(\R^d)$, with some $a \in (0,\infty)$. The authors also showed that
for a potential slower than $|x|^{-2}$, exponential decay of a zero-energy eigenfunction is possible. For some upper bounds
on the decay rates see \cite{H}. A more encompassing study of exponential decay has been made in \cite{HS}, in which also
decays faster than exponential have been ruled out.

Our framework for non-local Schr\"odinger operators with $\lambda_0 < 0$ discussed above does not extend to the zero
eigenvalue case, however, we will develop here a new framework by using a restricted set of operators $L$, which require
a doubling property of the jump kernels instead of the more general jump-paring property. Since now $\lambda_0 = 0$, there
is a complete lack of an energetic advantage from negative eigenvalues, however slight, leading to the behaviours discussed
above, and now the influence of the potential appears through vestigial effects resulting from its sign at infinity. This
will also reflect in the fact that in this case the contributions of the exits and the re-entries in local neighbourhoods
even out and add up to such an extent that \eqref{meanntime} obviously no longer holds, and the decay events become now
governed by global lifetimes such as
$$
\Lambda_{B(x,|x|/2)} = \ex^x \left[\int_0^{\tau_{B(x,|x|/2)}} e^{- \int_0^t V(X_s) ds} dt\right].
$$
From this expression it can be appreciated that in the $\lambda_0 = 0$ case there is a rather delicate difference between
the behaviour of paths under the perturbing potential and free fluctuations of the process, and this slight difference is
responsible even for the very existence of a ground state at zero eigenvalue (for further details see Section 6).

Below we start from the assumption that, for a class of operators $L$ and decaying $V$, the eigenvalue equation
$(-L+V)\varphi = 0$ is satisfied by a function $\varphi \in L^p(\R^d)$, for some $p > 0$, describing zero-energy
eigenfunctions (when $p=2$) or zero-resonances (when $p\neq 2$). Then we will study the asymptotic behaviour of
$\varphi(x)$ as $|x|\to\infty$ which, following from the choice of the input operators, has a pointwise decay to zero.
%As it will be seen below, the cases of potentials being positive or negative at infinity need separate discussions,
%the latter case being more difficult.
Our main results for asymptotically positive potentials are Theorems \ref{thm:upper_bound_pos}-\ref{thm:lower_bound_pos},
giving upper and lower bounds on $\varphi$ when these functions are positive or when they may have zeroes and distinct
nodal domains. For asymptotically negative potentials we have Theorem \ref{thm:upper_bound_neg}, giving upper bounds. We
note that our results apply to both zero-energy eigenfunctions and zero-resonances. As it will be seen in applications
to specific cases (Theorems \ref{thm:polynomial}-\ref{thm:layered}), these estimates perform remarkably well, giving
exact or close hits of the precise asymptotics in \eqref{exdecays} above.

The remainder of this paper is organized as follows. In Section 2 we introduce the class of non-local Schr\"odinger
operators considered, and briefly summarize some of the relevant properties of the jump processes used in their
Feynman-Kac representations. We also give an expression of the solutions of the eigenvalue equation using a path
integral representation, which will be a key formula used throughout below. In Section 3 we derive and prove some
self-improving upper and lower estimates on the solutions of related harmonic functions, on which our main conclusions
will rely. In Sections 4 and 5 we obtain the decay behaviours separately for potentials positive and negative at
infinity, respectively. In the concluding Section 6 we illustrate these results on specific examples, and discuss
some mechanisms lying behind these decays.

\section{Non-local Schr\"odinger operators and Feynman-Kac semigroups}

\subsection{Non-local Schr\"odinger operators and related random processes}
%\label{sec:Prel}
We start by a general remark on notation. A ball centered in $x \in \R^d$ and of radius $r > 0$ will be denoted by
$B(x,r)$. We write $a \wedge b = \min\{a,b\}$ and $a \vee b = \max\{a,b\}$. The notation $C(a,b,c,...)$ will be used
for a positive constant dependent on parameters $a,b,c,...$, dependence on the operator $L$ or, equivalently, on the
related L\'evy process $\pro X$ will be indicated by $C(L)$ and $C(X)$, while dependence on the dimension $d$ is
assumed without being stated explicitly. Since constants appearing in definitions and statements play a role, they
will be numbered $C_1, C_2, ...$ so that they can be tracked. We will also use the notation $f \asymp C g$ meaning
that $C^{-1} g \leq f \leq Cg$ with a constant $C \geq 1$, while $f \asymp g$ means that there is a constant $C \geq
1$ such that the latter holds. By $f \approx g$ we understand that $\lim_{|x| \to \infty} f(x)/g(x)=1$. In proofs
$c_1, c_2, ...$ will be used to denote auxiliary constants.

Consider the pseudo-differential operator $L$ with symbol $\psi$, defined by
\begin{align}
\label{def:gen}
\widehat{L f}(\xi) = -\psi(\xi) \widehat{f}(\xi), \quad \xi \in \R^d, \; f \in \Dom(L),
\end{align}
with dense domain $\Dom (L)= \{f \in L^2(\R^d): \psi \widehat f \in L^2(\R^d)\} \subset L^2(\R^d)$, and where the
hats denote Fourier transform. We assume the symbol to be of the form
\begin{align}
\label{eq:Lchexp}
\psi(\xi) = A \xi \cdot \xi + \int_{\R^d\backslash \left\{0\right\}} (1-\cos(\xi \cdot z)) \nu(dz),
\end{align}
where $A=(a_{ij})_{1\leq i,j \leq d}$ is a symmetric non-negative definite matrix, and $\nu$ is a symmetric
Radon measure on $\R^d \backslash \left\{0\right\}$, i.e., $\nu(E)= \nu(-E)$, for every Borel set $E \subset
\R^d \backslash \left\{0\right\}$, with the property that $\int_{\R^d} (1 \wedge |z|^2) \nu(dz) < \infty$.
In the present paper we assume throughout without further notice that the measure $\nu(dz)$ has infinite total
mass and it is absolutely continuous with respect to Lebesgue measure, i.e., $\nu(\R^d \backslash \left\{0\right\})
=\infty$ and $\nu(dz)= \nu(z)dz$, with $\nu(z) > 0$, $z\in\R^d$. For simplicity, we denote the density  also by
$\nu$.
It is a standard fact \cite{bib:J} that $-L$ is a positive, self-adjoint operator with core $C_0^{\infty}(\R^d)$,
and the expression
$$
L f(x) = \sum_{i,j=1}^d a_{ij} \frac{\partial^2 f}{\partial x_j \partial x_i} (x) +
\lim_{\varepsilon \downarrow 0} \int_{|y-x|>\varepsilon} (f(z)-f(x))\nu(z-x)dz, \quad x \in \R^d,
\; f \in C_0^{\infty}(\R^d),
$$
holds. Also, $\Spec (-L) = \Spec_{\rm ess} (-L) = [0,\infty)$.
%%%%%%%%%DIRICHLET%%%%%%%%%%%%%
\iffalse
One can associate with $L$ the Dirichlet form $(\cE, \Dom(\cE))$
\begin{align}
\label{def:qform}
\cE(f,g) = \int_{\R^d} \psi(\xi) \widehat{f}(\xi) \overline{\widehat{g}(\xi)} d\xi, \quad f, g \in \Dom(\cE)
= \left\{h \in L^2(\R^d): \int_{\R^d} \psi(\xi) |\widehat{h} (\xi)|^2 d\xi < \infty \right\}.
\end{align}
Whenever $f \in \Dom(L)$, we have $\cE(f,g)=(f,g)$.
\fi
%%%%%%%%%%%%%%%%%%%%%%%%%%%%%%%

Below we will use the symmetrization
\begin{align} \label{eq:Lchexpprof}
\Psi(r) = \sup_{|\xi| \leq r} \psi(\xi), \quad r>0,
\end{align}
of the symbol $\psi$. Let $H(r) = \frac{\left\|A\right\|}{r^2}+ \int_{\R^d\backslash \left\{0\right\}}
\left(1 \wedge \frac{|z|^2}{r^2}\right) \nu(dz)$. A combination of \cite[Rem. 4.8]{bib:Sch} and \cite[Sect. 3]{bib:Pru}
gives that
\begin{align}
\label{eq:PruitH}
C_1 H\left(\frac{1}{r}\right) \leq \Psi(r) \leq C_2 H\left(\frac{1}{r}\right), \;\; r>0,
\end{align}
holds with suitable constants $C_1 \in (0,1), C_2 > 1$, independent of $A$ and $\nu$. Also, it follows directly
that $r \mapsto H(r)$ is non-increasing and the doubling property $H(r) \leq 4 H(2r)$, $r>0$, holds, in particular,
$\Psi(2r) \leq 4 (C_2/C_1) \Psi(r)$, for all $r>0$.

Furthermore, let
$$
\cB = \left\{f \in C^2_{\rm c}(\R^d): f(x)=1 \ \text{for} \ x \in B(0,1/2),
f(x)=0 \ \text{for} \ x \in B(0,1)^c \ \text{and} \ 0 \leq f \leq 1\right\}.
$$
Then for $f_s(x) = f(x/s)$ with $f \in \cB$ and $s>0$, we have
\begin{eqnarray*}
\left\|L f_s \right\|_{\infty}
& \leq &
2 \nu(B(0,s)^c) + \frac{1}{s^2}\sup_{i,j=1,...,d} \left\|\frac{\partial^2 f}{\partial x_i \partial x_j}\right\|_{\infty}
\left(\left\|A\right\| + \int_{|y|  \leq s}|y|^2 \nu(dy)\right) \\
& \leq &
\left(2 \vee \sup_{i,j=1,...,d} \left\|\frac{\partial^2 f}{\partial x_i \partial x_j}\right\|_{\infty} \right) H(s),
\quad s >0. \nonumber
\end{eqnarray*}
Let $f_s(x) = f(x/s)$ for $s>0$, and denote
\begin{eqnarray}\label{def:C9a}
C_3(L, s) := \inf_{f \in \cB} \left\|L f_s \right\|_{\infty}.
\end{eqnarray}
From the above observations it follows that
\begin{align}\label{def:C9}
C_3(L,s) \leq \frac{1}{C_1} \left(2 \vee \inf_{f \in \cB} \sup_{i,j=1,...,d} \left\|\frac{\partial^2 f}{\partial x_i
\partial x_j}\right\|_{\infty}  \right) \Psi(1/s), \quad s >0.
\end{align}

A special feature of the operators of the form (\ref{def:gen}) with (\ref{eq:Lchexp}) is that they can be treated
by a path integral approach. For each choice of the matrix $A$ and the measure $\nu(dz) = \nu(z)dz$, the operator $L$
is the infinitesimal generator of an $\R^d$-valued rotationally symmetric L\'evy process, $d \geq 1$, on the space of
c\`adl\`ag paths (i.e., functions $[0,\infty) \to \R^d$ which are continuous from the right, having left limits). We
denote by $\pro X$ the L\'evy process generated by $L$, the probability measure of the process starting in $x \in \R^d$
by $\pr^x$, and expectation with respect to $\pr^x$ by $\ex^x$. It is a general fact that $\pro X$ is a strong Markov
process with respect to its natural filtration, and its characteristic function is given by
$$
\ex^0 \left[e^{i \xi \cdot X_t}\right] = e^{-t \psi(\xi)}, \quad \xi \in \R^d, \ t >0,
$$
where $\psi$ is the symbol of $-L$ as defined in (\ref{eq:Lchexp}) which, from a probabilistic perspective, is the
L\'evy-Khintchin formula for the class of L\'evy processes we consider. In this context, $A$ is the diffusion
matrix describing the Brownian component of the process $\pro X$, and $\nu(dz)$ is the jump measure (called L\'evy
measure) describing the jump component, thus the L\'evy triplet of the process is $(0,A,\nu)$.
%\quad \nu(x)   > 0.
%\begin{align} \label{eq:nuinf}
%\nu(\R^d \backslash \left\{0\right\})=\infty \quad \text{and} \quad \nu(dx)=\nu(x)dx, \quad \text{with}
%\quad \nu(x)   > 0.
%\end{align}
When $A \equiv 0$, the random process $\pro X$ is a purely jump process, otherwise it contains an independent Brownian
component.

The above properties jointly imply that $\pro X$ is a strong Feller process, or equivalently, its one-dimensional
distributions are absolutely continuous with respect to Lebesgue measure, i.e., there exist measurable functions
$p(t,x,y) = p(t,0,y-x)=: p(t,y-x)$, corresponding to transition probability densities, such that $\pr^0(X_t \in E)
= \int_E p(t,x)dx$, for every Borel set $E \subset \R^d$, see \cite[Th. 27.7]{bib:Sat}. Let $D \subset \R^d$ be an
open bounded set and consider the first exit time
$$
\tau_D = \inf\left\{t \geq 0: X_t \notin D\right\}
$$
from $D$. The transition probability densities $p_D(t,x,y)$ of the process killed on exiting $D$ are then given by
the Dynkin-Hunt formula
\begin{align}
\label{eq:HuntF}
p_D(t,x,y)= p(t,y-x) - \ex^x\left[p(t-\tau_D, y-X_{\tau_D}); \tau_D < t \right], \quad x , y \in D.
\end{align}
The Green function of the process $\pro X$ on $D$ is thus $G_D(x,y)= \int_0^{\infty} p_D(t,x,y) dt$, for all $x, y
\in D$.

We also recall that when $D \subset \R^d$ is a bounded open domain, the following formula due to Ikeda and
Watanabe holds \cite[Th. 1]{bib:IW}: for every $\eta>0$ and every bounded or non-negative Borel function $f$ on
$\R^d$ such that $\dist(\supp f, D) >0$, we have
\begin{align}
\label{eq:IWF}
\ex^x\left[e^{-\eta \tau_D} f(X_{\tau_D})\right] = \int_D \int_0^{\infty} e^{-\eta t} p_D(t,x,y) dt \int_{D^c} f(z)
\nu(z-y) dzdy, \quad x \in D.
\end{align}
Furthermore, by \cite[Rem. 4.8]{bib:Sch} we have
\begin{align}
\label{eq:gen_est_tau}
\ex^0[\tau_{B(0,r)}] \leq  \frac{C_{4}}{\Psi(1/r)}, \quad r    > 0,
\end{align}
with a constant $C_{4}$ independent of the process. For more details on L\'evy processes and their generators
we refer to \cite{bib:Sat,bib:J}.

In the remainder of the paper we will use the following class of L\'evy processes and related operators $L$.
\begin{definition}
\label{jumpdoubling}
Let $\pro X$ be a L\'evy process with L\'evy-Khintchin exponent $\psi$ as in \eqref{eq:Lchexp} and L\'evy triplet
$(0,A,\nu)$, satisfying the following conditions.
\begin{itemize}
\item[\textbf{(A1)}] %\texttt{L\'evy intensity:}
There exist a non-increasing function $g:(0,\infty) \to (0,\infty)$ and constants $C_5$, $C_6 > 0$ such that
$$
\nu(x) \asymp C_5 g(|x|), \quad x \in \R^d \backslash \left\{0\right\}, \quad \text{and} \quad g(|x|) \leq
C_6 g(2|x|),  \quad |x| \geq 1.
$$
\item[\textbf{(A2)}] %\texttt{Transition density:}
There exists $t_{\rm b} >0$ such that $\sup_{x\in \R^d} p(t_{\rm b},x) = p(t_{\rm b},0) < \infty$.
\medskip
\item[\textbf{(A3)}] %\texttt{Green function:}
There exists a constant $C_7=C_7(X)$ such that
$$
\sup_{x,y: \, |x-y| \geq s/8 } G_{B(0,s)}(x,y) \leq C_7 \frac{\Psi(1/s)}{s^d}, \quad s \geq 1.
$$
\end{itemize}

\end{definition}
\noindent
Assumption (A1) is a statement on the profile of $\nu$, including a doubling property. Assumption (A2)
is equivalent with $e^{-t_{\rm b} \psi} \in L^1(\R^d)$, for some $t_{\rm b}  >0$, which
%In this case $p(t_{\rm b},x)$ can be obtained by the Fourier inversion formula. Clearly,
by the Markov property of $\pro X$ extends to every $t >t_{\rm b}$.
%For more details on the existence and properties of transition probability densities for L\'evy processes we
%refer to \cite{bib:KSch} and references therein. Finally, we remark that in many cases of interest
%Assumption (A3) follows directly from (rough) space-time estimates of the densities $p(t,x)$. For unimodal...
Assumption (A3) is a technical condition on the Green function for particular balls. The class of processes
satisfying Assumptions (A1)-(A3) includes isotropic and anisotropic stable processes (corresponding to fractional
Schr\"odinger operators), and layered stable processes (corresponding to another class of L\'evy-operators), which
will be discussed in Section 6.
% large subclasses of isotropic unimodal L\'evy processes,
%subordinate Brownian motions, L\'evy processes with non-degenerate Brownian components, as well as symmetric
%stable-like processes.
%Equivalently, in terms of operators $L$ this covers fractional Laplacians, . . .

Throughout this paper we will use \emph{$X$-Kato class potentials}. We say that the Borel function $V: \R^d\to\R$,
called potential, belongs to $\cK^X$ associated with the L\'evy process $\pro X$ if it satisfies
\begin{align}
\label{eq:Katoclass}
\lim_{t \downarrow 0} \sup_{x \in \R^d} \ex^x \left[\int_0^t |V(X_s)| ds\right] = 0.
\end{align}
Also, we say that $V = V_+-V_-$ is in $X$-Kato class whenever its positive and negative parts satisfy $V_- \in \cK^X$
and $V_+ \in \cK^X_{\loc}$, where $V_+ \in \cK^X_{\loc}$ means that $V_+ 1_B \in \cK^X$ for all compact $B \subset
\R^d$. It is straightforward to see that $L^{\infty}_{\loc}(\Rd) \subset \cK_{\loc}^X$, and by stochastic continuity
of $\pro X$ also $\cK_{\loc}^X \subset L^1_{\loc}(\R^d)$. Note that condition \eqref{eq:Katoclass} allows local
singularities of $V$.

\iffalse
We now give the class of potentials which will be used in this paper.
\begin{definition}[\textbf{$X$-Kato class}]
{\rm
We say that the Borel function $V: \R^d \to \R$ called \emph{potential} belongs to \emph{Kato-class} $\cK^X$ associated
with the L\'evy process $\pro X$ if it satisfies
\begin{align}
\label{eq:Katoclass}
\lim_{t \downarrow 0} \sup_{x \in \R^d} \ex^x \left[\int_0^t |V(X_s)| ds\right] = 0.
\end{align}
Also, we say that $V$ is an \emph{$X$-Kato decomposable potential}, denoted $V \in \cK^X_{\pm}$, whenever
$$
V=V_+-V_-, \quad \text{with} \quad V_- \in \cK^X \quad \text{and} \quad V_+ \in \cK^X_{\loc},
$$
where $V_+$, $V_-$ denote the positive and negative parts of $V$, respectively, and where $V_+ \in \cK^X_{\loc}$
means that $V_+ 1_B \in \cK^X$ for all compact sets $B \subset \R^d$.
\label{Xkato}
}
\end{definition}
\noindent
In what follows, for $X$-Kato decomposable potentials we will use the simpler name $X$-Kato class potentials.
\fi

%For specific processes $\pro X$ the definition of $X$-Kato class can be explicitly  reformulated in an analytic way
%in terms of the kernel $p(t,x)$ restricted to small $t$ and small $x$. It is shown in \cite[Cor. 1.3]{bib:GrzSz} that
%\eqref{eq:Katoclass} is equivalent with
%\begin{align} \label{eq:Kato_new}
%\lim_{t \to 0^{+}} \sup_{x \in \R^d} \int_0^t \int_{B(x,t)} p(s,x-y)|V(y)| dy ds = 0.
%\end{align}
%This condition can be directly checked by using available two-sided sharp estimates of the kernel $p(t,x)$
%in small time (see e.g. \cite{bib:BGR, bib:KS14, bib:KS13}).

With an operator $L$ given by (\ref{def:gen}) and an $X$-Kato class potential $V$, viewed as a multiplication
operator, we call the operator
\begin{equation}
\label{nonlocSch}
H = -L + V \quad \text{in} \quad L^2(\R^d)
\end{equation}
defined by form-sum a \emph{non-local Schr\"odinger operator}. To study the spectral properties of this operator, we
use a Feynman-Kac type representation.
%This representation will also allow us to give a self-adjoint realization of $H$.

\iffalse
\begin{definition}[\textbf{Non-local Schr\"odinger operator}]
Let $L$ be an operator given by (\ref{def:gen}), and $V$ be an $X$-Kato class potential viewed as a multiplication
operator. We call the operator
\begin{equation}
\label{nonlocSch}
H = -L + V
\end{equation}
defined by form sum a \emph{non-local Schr\"odinger operator}.
\end{definition}
\noindent
\fi

\subsection{Feynman-Kac semigroups and generalized eigenfunctions}
\noindent
Consider the one-parameter family of operators
$$
T_t f(x) = \ex^x\left[e^{-\int_0^t V(X_s) ds} f(X_t)\right], \quad \ t>0.
$$
By standard arguments based on Khasminskii's Lemma, see \cite[Lem. 3.37-3.38]{LHB}, for an $X$-Kato class
potential $V$ it follows that there exist constants $C_8(X,V), C_9(X,V) > 0$ such that
\begin{align}
\label{eq:khas}
\sup_{x \in \R^d} \ex^x\left[e^{-\int_0^t V(X_s)ds}\right] \leq \sup_{x \in \R^d}
\ex^x\left[e^{\int_0^t V_-(X_s)ds}\right] \leq C_8 e^{C_9t}, \quad t>0.
\end{align}
This implies that the operators $T_t$, $t > 0$, are well defined on every $L^p(\R^d)$, $1 \leq p \leq \infty$, and
$$
\left\|T_t\right\|_{p,p} \leq \left\|T_t\right\|_{\infty,\infty} \leq  C_8 e^{C_9t}, \quad t>0.
$$
Moreover, the family $\{T_t: t\geq 0\}$ is a strongly continuous semigroup of operators on each
$L^p(\R^d)$, $1 \leq p \leq \infty$, which we call the \emph{Feynman-Kac semigroup} associated with the process
$\pro X$ and potential $V$. We define
$$
H f = -\slim_{t \downarrow 0} \frac{T_t f - f }{t},
$$
for those functions $f \in L^p(\R^d)$ for which the limit exists. We denote the set of all such functions by $\Dom_{L^p} H$
and call it the $L^p$-domain on $H$. It is known that $H$ is a closed unbounded operator such that $\Dom_{L^p} H$ is dense
in $L^p(\R^d)$. For $p=2$ the operator $H$ can be identified with a self-adjoint operator as given by (\ref{nonlocSch}),
defined in a quadratic form sense \cite[Ch. 2]{bib:DC}. A specific class of non-local Schr\"odinger operators given by $H =
\Phi(-\Delta) + V$, where $\Phi$ is a so called Bernstein function, has been defined and studied in \cite{bib:HIL12,bib:HL12}.

Next we summarize some basic properties of the operators $T_t$ which will be useful below. Recall that for a function $f \in
L^2(\R^d)$ we write $f \geq 0$ (resp., $f > 0$) if $f(x) \geq 0$ (resp., $f(x) > 0$) almost everywhere in $\R^d$.
\begin{lemma}
\label{lm:semprop}
Let $\pro X$ be a symmetric L\'evy process with L\'evy-Khintchin exponent satisfying \eqref{eq:Lchexp} such that Assumption
(A2) holds with some $t_{\rm b} > 0$, and let $V$ be an $X$-Kato class potential. Then the following properties hold:
\begin{itemize}
\item[(1)]
For all non-negative Borel measurable functions $f, g$ we have
$$
\int_{R^d} f(x) T_t g(x) dx = \int_{\R^d} T_t f(x) g(x) dx, \quad t>0,
$$
i.e., $T_t$ are symmetric operators.
\item[(2)]
The operators $T_t: L^p(\R^d) \to L^{\infty}(\R^d)$ for $1 < p \leq \infty$, $t \geq t_{\rm b}$,
and
$T_t: L^1(\R^d) \to L^{\infty}(\R^d)$ for $t \geq 2t_{\rm b}$, are bounded.
\vspace{0.1cm}
\item[(3)]
For all $t \geq 2t_{\rm b}$, $T_t$ has a bounded measurable integral kernel $u(t, x, y)$, symmetric in $x$ and $y$,
i.e., $T_t f(x) = \int_{\R^d} u(t,x,y)f(y) dy$, for all $f \in L^p(\R^d)$ and $1 \leq p \leq \infty$.
\vspace{0.1cm}
\item[(4)]
For all $t>0$ and $f \in L^{\infty}(\R^d)$, $T_t f$ is a bounded continuous function, i.e., $\{T_t: t\geq 0\}$
is a strongly Feller semigroup.
\vspace{0.1cm}
\item[(5)]
For all $t > 0$ the operators $T_t$ are positivity improving, i.e., $T_t f > 0$ for all $f \in L^2(\R^d)$ such
that $f \geq 0$ and $f \neq 0$ a.e.
\end{itemize}
\end{lemma}
\noindent
The proof of these general properties is left to the reader, which can be obtained as an extension to the
present set-up of the facts in \cite[Sects. 3.2-3.3]{LHB}. Note that we do not assume that $p(t,x)$ is bounded
for all $t>0$, and thus in general the operators $T_t: L^p(\R^d) \to L^{\infty}(\R^d)$ need not be bounded for
$t<t_{\rm b}$.

Related to the Feynman-Kac semigroup, we define the potential operator by
\begin{align*}
G^V f(x) = \int_0^\infty T_t f(x) dt = \ex^x \left[\int_0^\infty e^{-\int_0^t V(X_s)ds} f(X_t) dt \right],
\end{align*}
for non-negative or bounded Borel functions $f$ on $\R^d$. Recall that $\tau_D$ denotes the first exit time of
the process from domain $D$. Whenever $D \subset \R^d$ is an open set, it follows by the strong Markov property
of the process that for every $x \in D$
\begin{equation}
\label{eq:pot1}
\begin{split}
G^V f(x)
& =
\ex^x\left[\int_0^{\tau_D} e^{-\int_0^t V(X_s)ds} \, f(X_t) dt \right] + \ex^x\left[e^{-\int_0^{\tau_D} V(X_s)ds}
\, G^V f(X_{\tau_D}); \tau_D < \infty \right].
\end{split}
\end{equation}
For further information on potential theory we refer to \cite{bib:BKK}.

Let $V$ be a \emph{decaying} $X$-Kato class potential by which we mean $V(x) \to 0$ as $|x|\to\infty$. The main object
of our investigations in this paper are the solutions $\varphi \in L^p(\R^d)$, $p \geq 1$, $\varphi \not \equiv 0$, of
the equation
\begin{equation}
\label{eveq}
H\varphi = 0,
\end{equation}
or, equivalently,
\begin{align}
\label{eq:eign}
T_t \varphi = \varphi, \quad t \geq 0.
\end{align}
In the $L^2$-framework
$$
\Sigma := \inf \Spec_{\rm ess} H = \inf \Spec_{\rm ess} (-L) = 0
$$
is the edge of the essential spectrum of $H$ and \eqref{eveq} can be understood as the eigenvalue equation at $\Sigma$. 
Whenever the solution $\varphi$ to \eqref{eveq} is such that $\varphi \in \Dom_{L^2} H$, we call it a \emph{zero-energy 
eigenfunction} (or \emph{zero-energy bound state}) and then 0 is an eigenvalue. Otherwise, we call both (by a slight 
abuse of language) a \emph{zero-resonance}.

Throughout we will assume that every zero-energy eigenfunction $\varphi$ is $L^2$-normalized so that $\|\varphi\|_{2} = 1$. 
Moreover, by Lemma \ref{lm:semprop} we have $T_t(L^p(\R^d)) \subset L^{\infty}(\R^d)$ and $T_t(L^{\infty}(\R^d)) \subset 
C_{\rm b}(\R^d)$ for every $t>2t_{\rm b}$ and $p \geq 1$. Therefore, any solution $\varphi$ to \eqref{eveq}-\eqref{eq:eign} 
is a bounded and continuous function, in particular, it makes sense to study their pointwise estimates.

Below we will make frequent use the following resolvent representation of solutions of (\ref{eveq}). Choose $\theta >0$. 
Then by multiplying both sides of (\ref{eq:eign}) and integrating with respect to time, we obtain
\begin{align}\label{eq:eig}
\varphi(x) = \theta
\int_0^{\infty}\ex^x\left[ e^{-\int_0^t(\theta+V(X_s))ds} \varphi(X_t) \right]dt, \quad x \in \R^d.
\end{align}
Combining this with \eqref{eq:pot1} applied to $f=\varphi$ for an arbitrary open set $D \subset \R^d$ and $x \in D$, and
using the strong Markov property of the L\'evy process $\pro X$, we readily obtain
\begin{eqnarray}
\label{eq:eig1}
\varphi(x)
&=&
\theta \left(\int_0^{\tau_D} + \int_{\tau_D}^\infty\right)\ex^x\left[ e^{-\int_0^t(\theta+V(X_s))ds} \varphi(X_t) \right]dt
\nonumber\\
&=&
\theta \ex^x\left[ \int_0^{\tau_D} e^{-\int_0^t(\theta+V(X_s))ds} \varphi(X_t)
dt\right] + \ex^x\left[e^{-\int_0^{\tau_D}(\theta+V(X_s))ds} \varphi(X_{\tau_D}) ; \tau_D < \infty \right],
\end{eqnarray}
which will be a fundamental formula in what follows. %This representation holds equally for $\varphi \in L^p(\R^d)$, $p \geq 1$.

\bigskip

\section{Self-improving estimates}
\label{sec:tech}
In this section we show some key estimates %on $H$-harmonic functions
which will serve to proving our main results below.
Let $R_0 \geq 1$ be a fixed number. Throughout this section we will consider non-increasing functions
$u, v, w :[R_0,\infty) \to (0,\infty)$ such that there exist $C_{10}, C_{12} \geq 1$ satisfying
\begin{align} \label{eq:doubl_prop}
u(r) \leq C_{10} u(2r) \qquad w(r) \leq C_{12} w(2r), \quad r \geq R_0,
\end{align}
and
\begin{align} \label{eq:dom}
\int_{|x| > r} u(|x|) \leq  w(r), \quad r \geq R_0.
\end{align}
In what follows we will also use the notations
$$
K_{u,v}:=\frac{u}{v} \qquad \text{and} \qquad h_{u,v}(r)
:= \int_{R_0 \leq |y| \leq r} K_{u,v}(|y|) dy, \quad r \geq R_0,
$$
and $\omega_d$ will denote the volume of a unit ball.

%The following lemma is our first main technical result of this paper and it will be fundamental for our investigations below.

\begin{lemma}
\label{lem:defic1}
Let $u, v, w :[R_0,\infty) \to (0,\infty)$ be non-increasing functions such that \eqref{eq:doubl_prop}-\eqref{eq:dom}
hold with constants $C_{10}, C_{12} \geq 1$ and
\begin{align}
\label{eq:trans}
\lim_{r \to \infty} \frac{w(r)}{v(r)} = 0.
\end{align}
Moreover, suppose that $f$ is a bounded non-negative function on $\R^d$ such that
\begin{align}
\label{eq:1st_iteration}
f(x) \leq \frac{C_{13}}{v(|x|)} \left(\int_{|z-x| > \frac{|x|}{2}}f(z) u(|z-x|) dz + \frac{w(|x|)}{|x|^d}
\int_{\frac{|x|}{32} < |z-x| \leq \frac{|x|}{2}} f(z) dz  \right), \quad |x| \geq R_0,
\end{align}
for a constant $C_{13} >0$, and let $\eta:=C_{10}C_{13}$. If
\begin{align} \label{eq:bound_h}
\sup_{r>R_0} K_{u,v}(r) e^{\eta h_{u,v}(r)} < \infty,
\end{align}
then there exist $R > 2R_0$ and $C_{14}=C_{14}(R)$ such that
\begin{align}
\label{eq:1st_iteration_h}
f(x) \leq C_{14} \left\|f\right\|_{\infty} K_{u,v}(|x|) e^{\eta h_{u,v}(|x|)}, \quad |x| \geq R.
\end{align}
In particular, if
$$
\int_{|y| \geq R_0} K_{u,v}(|y|) dy < \infty,
$$
then
\begin{align} \label{eq:1st_iteration_hh_2}
f(x) \leq C_{14} e^{\eta h_{u,v}(\infty)} \left\|f\right\|_{\infty} K_{u,v}(|x|) ,
\quad |x| \geq R, \quad \text{with} \quad h_{u,v}(\infty):= \lim_{r \to \infty} h_{u,v}(r).
\end{align}
\end{lemma}

\begin{proof}
Observe that \eqref{eq:1st_iteration_hh_2} follows directly from \eqref{eq:1st_iteration_h}. We only
need to prove \eqref{eq:1st_iteration_h}. Let
\begin{align} \label{eq:crucial_1_h}
c_1 := \sup_{r \geq R_0} K_{u,v}(r) e^{\eta h_{u,v}(r)}
\end{align}
and $R >2R_0$ be large enough such that
\begin{align} \label{eq:crucial_2_h}
c_2 := \sup_{|x| \geq R}  \frac{w(|x|)}{v(|x|)} < \frac{1}{(1+c_1) C_{13}\big(C_{12} + 2^{-d}\big)}.
\end{align}
By \eqref{eq:1st_iteration} and the part of \eqref{eq:doubl_prop} stated for $u$, for every $|x| \geq R$
we have
\begin{eqnarray*}
f(x)
& \leq &
\frac{C_{13}}{v(|x|)} \left(\int_{|z-x| > \frac{|x|}{2} \atop |z| \leq R }f(z) u(|z-x|) dz + \frac{w(|x|) \,
\1_{\left\{|x| \leq 2R\right\}}}{|x|^d} \int_{\frac{|x|}{32} < |z-x| \leq \frac{|x|}{2}} f(z) dz  \right)
\nonumber \\
& \qquad + &
\frac{C_{13}}{v(|x|)} \left(\int_{|z-x| > \frac{|x|}{2} \atop |z| > R }f(z) u(|z-x|) dz + \frac{w(|x|) \,
\1_{\left\{|x| > 2R\right\}}}{|x|^d} \int_{\frac{|x|}{32} < |z-x| \leq \frac{|x|}{2}} f(z) dz  \right) \nonumber \\
& \leq &
\frac{C_{13} \, \|f\|_\infty}{v(|x|)} \left(\omega_d R^d u(|x|/2) + \frac{\omega_d \, w(2R)}{2^d u(2R)}u(|x|)\right)
\nonumber \\
& \qquad + &
\frac{C_{13}}{v(|x|)} \left(\int_{|z-x| > \frac{|x|}{2} \atop |z| > R }f(z) u(|z-x|) dz + \frac{w(|x|) \,
\1_{\left\{|x| > 2R\right\}}}{|x|^d} \int_{\frac{|x|}{32} < |z-x| \leq \frac{|x|}{2}} f(z) dz  \right) \nonumber \\
& \leq &
C_{13} \,\omega_d \left( C_{10} R^d+ \frac{w(2R)}{2^d u(2R)}\right) \|f\|_\infty \, \frac{u(|x|)}{v(|x|)} \nonumber \\
& \qquad + &
\frac{C_{13}}{v(|x|)} \left(\int_{|z-x| > \frac{|x|}{2} \atop |z| > R }f(z) u(|z-x|) dz +
\frac{w(|x|)}{|x|^d} \int_{\frac{|x|}{32} < |z-x| \leq \frac{|x|}{2} \atop |z| > R} f(z) dz  \right).
\end{eqnarray*}
From this we obtain the two independent estimates
\begin{align*}
f(x)
& \leq
C_{13} \,\omega_d \left( C_{10} R^d+ \frac{w(2R)}{2^d u(2R)}\right) \|f\|_\infty \, \frac{u(|x|)}{v(|x|)}  \\
& \quad +
\frac{C_{13}}{v(|x|)} \int_{R \leq |z| \leq |x| \atop |x-z| > \frac{|x|}{2}} f(z) u(|x-z|) dz +
C_{13}\frac{w(|x|)}{v(|x|)} \left(\frac{\int_{|z| > \frac{|x|}{2}}u(|z|) dz}{w(|x|)} + \frac{1}{2^d} \right)
\sup_{|z| \geq |x|/2 \vee R} f(z)
\end{align*}
and
\begin{align*}
f(x)
\leq
C_{13} \,\omega_d \left( C_{10} R^d+ \frac{w(2R)}{2^d u(2R)}\right) \|f\|_\infty \, \frac{u(|x|)}{v(|x|)} + C_{13}
\|f\|_\infty \frac{w(|x|)}{v(|x|)} \left(\frac{\int_{|z| > \frac{|x|}{2}}u(|z|) dz}{w(|x|)} + \frac{1}{2^d}\right).
\end{align*}
They give, respectively,
\begin{align} \label{eq:proc1_h}
f(x) \leq c_3 \left\|f\right\|_{\infty} \frac{u(|x|)}{v(|x|)}
+ \frac{C_{13}}{v(|x|)} \int_{R \leq |z| \leq
|x| \atop |x-z| > \frac{|x|}{2}} f(z) u(|x-z|) dz +
c_2 \, C_{13} \left(C_{12} + \frac{1}{2^d}\right) \sup_{|z|\geq |x|/2 \vee R} f(z)
\end{align}
and
\begin{align} \label{eq:proc2_h}
f(x) & \leq  c_3\|f\|_{\infty} \, \left( \frac{u(|x|)}{v(|x|)} +
c_2  C_{13}\big(C_{12} + 2^{-d}\big)\right),
\end{align}
for $|x| \geq R$, where
$$
c_3 := C_{13} \,\omega_d \left(C_{10} R^d+ \frac{w(2R)}{2^d u(2R)}\right) \vee 1.
$$
Also, recall that $\eta = C_{10} C_{13}$, denote $c_4 = c_2 (1+c_1) C_{13}\big(C_{12} + 2^{-d}\big)$, and notice
that by \eqref{eq:crucial_1_h}-\eqref{eq:crucial_2_h} we have $c_4 < 1$.

Now we show that for every $p \in \N$
\begin{align} \label{eq:eq_ind_h}
f(x) \leq
c_3 \left\|f\right\|_{\infty} \left[K_{u,v}(|x|) \sum_{k=1}^p \frac{\big(\eta h_{u,v}(|x|) \big)^{k-1}}{(k-1)!} +
c_4^p \right], \quad |x| \geq R.
\end{align}
Notice that if this holds, then by taking the limit $p \to \infty$ it follows that
$$
f(x) \leq c_3 \left\|f\right\|_{\infty} K_{u,v}(|x|) e^{\eta h_{u,v}(|x|)}, \quad |x| \geq R,
$$
which is the bound stated in the lemma.
To prove \eqref{eq:eq_ind_h} we make induction on $p \in \N$. First observe that for $p=1$ the estimate
\eqref{eq:eq_ind_h} follows from \eqref{eq:proc2_h}. Suppose now that \eqref{eq:eq_ind_h} holds for $p-1 \in \N$.
By using \eqref{eq:proc1_h} and the induction hypothesis, we see for all $|x| \geq R$ that
\begin{align*}
f(x) & \leq c_3 \left\|f\right\|_{\infty} K_{u,v}(|x|) \\
& \ \
+ c_3 C_{10}C_{13}\left\|f\right\|_{\infty}  K_{u,v}(|x|) \sum_{k=1}^{p-1} \frac{\eta^{k-1}}{(k-1)!}
\int_{1 \leq |z| \leq |x|} K_{u,v}(|z|)h_{u,v}(|z|)^{k-1} dz \\
&  \ \
+ c_3 \left\|f\right\|_{\infty}\left(c_2 C_{12} C_{13} + c_1 c_2 C_{13} (C_{12}+ 2^{-d}) +
c_2 C_{13} (C_{12}+ 2^{-d})\right) c_4^{p-1}.
\end{align*}
By the substitution $h_{u,v}(r) = u$ we obtain
\begin{align*}
\int_{R_0 \leq |z| \leq |x|} K_{u,v}(|z|)h_{u,v}(|z|)^{k-1} dz
& = \int_{R_0}^{|x|} K_{u,v}(r) r^{d-1} h_{u,v}(r)^{k-1} dr  \\
& = \int_{h_{u,v}(R_0)}^{h_{u,v}(|x|)} u^{k-1} du = \frac{h_{u,v}(|x|)^k - h_{u,v}(R_0)^k}{k},
\end{align*}
and thus
\begin{align*}
f(x)
& \leq c_3 \left\|f\right\|_{\infty} K_{u,v}(|x|) + c_3 \left\|f\right\|_{\infty}  K_{u,v}(|x|)
\sum_{k=1}^{p-1} \frac{\eta^{k}}{k!} h_{u,v}(|x|)^k + c_3 \left\|f\right\|_{\infty} c_4^p \\
& =
c_3 \left\|f\right\|_{\infty} \left[K_{u,v}(|x|)
\sum_{k=1}^p \frac{\big(\eta h_{u,v}(|x|) \big)^{k-1}}{(k-1)!} + c_4^p \right] , \qquad |x| \geq R,
\end{align*}
which is the claimed bound.
\end{proof}

It is direct to check that under the non-restrictive assumption that the function $K_{u,v}$ is almost
non-increasing, i.e., there exists $C \geq 1$ such that $K_{u,v}(s) \leq C K_{u,v}(r)$ for all $R_0 \leq r \leq s$,
the integrability condition
$$
\int_{|z| \geq R_0} K_{u,v}(|z|)dz < \infty
$$
is equivalent with the convolution condition
$$
\sup_{|x| \geq R_0} \frac{\int_{|z-x| > R_0, |z| >R_0} K_{u,v}(|x-z|)K_{u,v}(|z|)dz}{K_{u,v}(|x|)} < \infty.
$$

The next lemma deals with a lower bound on positive functions satisfying an integral inequality.
\begin{lemma}
\label{lem:lower_tech}
Let $u, v:[1,\infty) \to (0,\infty)$ be non-increasing functions such that $u$ satisfies \eqref{eq:doubl_prop}
with a constant $C_{10} \geq 1$, and suppose that $f$ is a positive function on $\R^d$ such that
\begin{align}
\label{eq:1st_iteration_low}
f(x) \geq \frac{C_{15}}{v(|x|)} \int_{|z| < |x| \atop |z+x| < |z-x|}f(z) u(|z-x|) dz, \quad |x| > 1,
\end{align}
for a constant $C_{15} >0$. Then
\begin{align}
\label{eq:1st_iteration_h_low}
f(x) \geq \eta e^{-\eta h_{u,v}(1)}K_{u,v}(|x|) e^{\eta h_{u,v}(|x|)}, \quad |x| > 1,
\end{align}
with constant
$$
\eta:=
\frac{C_{15}}{C_{10}} \,\left(\frac{1}{2} \wedge \inf_{x: |x|=1} \int_{|z|\leq 1 \atop |z+x| < |z-x|}f(z) dz\right).
$$
In particular,
\begin{align}
\label{eq:1st_iteration_h_2}
f(x) \geq \eta K_{u,v}(|x|) , \quad |x| > 1.
\end{align}
\end{lemma}

\begin{proof}
By \eqref{eq:1st_iteration_low} we have
\begin{align*}
f(x) \geq \frac{C_{15}}{C_{10}} \frac{u(|x|)}{v(|x|)} \left(\int_{|z| \leq 1 \atop |z+x| < |z-x|} f(z) dz +
\int_{1 < |z| < |x| \atop |z+x| < |z-x|} f(z) dz\right),  \quad |x| > 1,
\end{align*}
and by symmetrization of the second integral,
\begin{align} \label{eq:low1}
f(x) \geq \frac{C_{15}}{C_{10}} K_{u,v}(|x|)  \left(\int_{|z| \leq 1 \atop |z+x| < |z-x|} f(z) dz +
\frac{1}{2}\left(\int_{1 < |z| < |x| \atop |z+x| < |z-x|} f(z) dz +
\int_{1 < |z| < |x| \atop |z-x| < |z+x|} f(-z) dz\right)\right),
\end{align}
for all $|x| > 1$. Next we prove that for every $p \in \N$
\begin{align} \label{eq:low2}
f(x) \geq \eta K_{u,v}(|x|) \sum_{k=1}^{p} \frac{\big(\eta (h_{u,v}(|x|)-h_{u,v}(1))\big)^{k-1}}{(k-1)!},
\quad |x| > 1,
\end{align}
holds with
$$
\eta:=\frac{C_{15}}{C_{10}} \,\left(\frac{1}{2} \wedge \inf_{|x| = 1} \int_{|z| \leq 1 \atop |z+x| < |z-x|}f(z) dz\right).
$$
Clearly, if \eqref{eq:low2} is true for every $p \in \N$, then estimate \eqref{eq:1st_iteration_h_low} also holds.

We use induction on $p$. For $p=1$ the inequality \eqref{eq:low2} is an immediate consequence of \eqref{eq:low1}. Suppose
now that the induction hypothesis holds for some $p \in \N$. By \eqref{eq:low1}-\eqref{eq:low2} and rotation symmetry
we have
\begin{align*}
f(x) \geq \eta K_{u,v}(|x|) \left(1 + \eta \sum_{k=1}^{p} \frac{\eta^{k-1}}{(k-1)!} \int_{1 < |z| <|x|}
\frac{u(|z|)}{v(|z|)}(h_{u,v}(|z|)-h_{u,v}(1))^{k-1} dz\right),  \quad |x| > 1.
\end{align*}
Since
\begin{align*}
\int_{1 < |z| <|x|} \frac{u(|z|)}{v(|z|)}(h_{u,v}(|z|)-h_{u,v}(1))^{k-1} dz
& =
\int_1^{|x|} \frac{u(r)}{v(r)}(h_{u,v}(r)-h_{u,v}(1))^{k-1} r^{d-1} \\
& = \frac{(h_{u,v}(|x|)-h_{u,v}(1))^{k}}{k}, \quad |x|> 1,
\end{align*}
we conclude that
\begin{align*}
f(x) & \geq \eta K_{u,v}(|x|) \left(1 + \sum_{k=1}^{p} \frac{\eta^{k}}{k!} (h_{u,v}(|x|)-h_{u,v}(1))^{k}\right) \\
& = \eta K_{u,v}(|x|) \sum_{k=1}^{p+1} \frac{\eta^{k-1}}{(k-1)!} (h_{u,v}(|x|)-h_{u,v}(1))^{k-1},  \quad |x| > 1,
\end{align*}
as required.
\end{proof}

\section{Decay of zero-energy eigenfunctions for potentials positive at infinity}

\subsection{Upper bound}
\noindent
Now we turn to discussing the spatial decay properties of eigenfunctions of non-local Schr\"odinger operators
presented in Section 2. %Except for the last section (****MIXED POTENTIALS****),
In this section we consider \emph{decaying potentials that are non-negative at infinity} in the following sense:

\medskip

\begin{itemize}
\item[\textbf{(A4)}]
$V$ is
\iffalse
Let $V \in \cK_{\pm}^X$ be
\fi
an $X$-Kato class potential such that $V(x) \to 0$ as $|x| \to \infty$, and there exists $r_0 >0$ such that
$V(x) \geq 0$ for $|x| \geq r_0$.
\end{itemize}

\medskip

It will be useful to introduce the notation
$$
V_{*}(x):= \inf_{r_0 \leq |y| \leq \frac{3}{2}|x|} V(y), \quad |x| \geq r_0.
$$
Notice that $V_{*}(x)$ is a radial and non-increasing function such that $V_{*}(x) \geq 0$, $|x| \geq r_0$.

We will need a uniform estimate of functions that are harmonic with respect to the operator $H$. Since our approach
is via a Feynman-Kac type stochastic representation, we use throughout the following probabilistic definition.
Let $D$ be an open subset of $\R^d$ and let $V$ be a Kato-class potential such that $V(x) \geq 0$ on $D$. We call a
non-negative Borel function $f$ on $\R^d$ an \emph{$(X,V)$-harmonic function} in the domain $D$ if
\begin{align}
\label{def:harm}
f(x) & = \ex^x\left[e^{-\int_0^{\tau_U} V(X_s)ds }  f(X_{\tau_U}) ; \tau_U < \infty \right], \quad x \in U,
\end{align}
for every open set $U$ with its closure $\overline{U}$ contained in $D$, and a \emph{regular $(X,V)$-harmonic function}
in $D$ if \eqref{def:harm} holds for $U=D$ (where $\tau_U$ is the first exit time from $U$). By the strong Markov
property every regular $(X,V)$-harmonic function in $D$ is $(X,V)$-harmonic in $D$. Whenever $V \equiv 0$ in $D$,
we refer to $f$ as a \emph{(regular) $X$-harmonic function}.

An initial version of the type of bound we prove below has been first obtained in \cite[Lem. 3.1]{KL17} and it can be
derived from the general results in \cite{bib:BKK}. Here we need a variant suitable for the purposes of the present
paper. Note that the following estimate does not exclude the case $V \equiv 0$.

\begin{lemma}
\label{lm:bhi}
Let $(X_t)_{t \geq 0}$ be a  L\'evy process with L\'evy-Khintchin exponent $\psi$ as in \eqref{eq:Lchexp} such
that Assumptions (A1)-(A3) and (A4) hold; specifically, let (A4) hold with some $r_0 > 0$. Then for every $\eta \in
(0,\frac{1}{4}]$ there exists a constant $C_{16} >0$ such that for every non-negative function $f$ on $\R^d$ which
is regular $(X,V)$-harmonic in a ball $B\big(x,\eta|x|\big)$, $|x| \geq r_0/(1-\eta)$, we have
\begin{align}
\label{eq:harmest}
f(y)\leq \frac{C_{16}}{V_{*}(x) \vee \Psi\big(\frac{1}{|x|}\big)} \left(\int \limits_{|z-x| >
2\eta|x|}f(z) \nu(z-x)dz +
\frac{\Psi\big(\frac{1}{|x|}\big)}{|x|^d}\int \limits_{\frac{\eta|x|}{8} < |x-z| \leq 2\eta|x|} f(z)dz \right),
\end{align}
whenever $|x-y| < \frac{\eta|x|}{32}$.
\end{lemma}

\begin{proof}
Following \cite[Sect. 3.2]{KL17}, for $s_1\geq 1$ and $s_2 = 2s_1$ define
\begin{align*}
h_1(s_1,s_2) = K_2(s_1, s_2,\infty) \left[C_3\left(L,\frac{s_1}{16}\right) \Big(C_{17}(s_1)
\, |B(0,s_1)|  + \ex^0 [\tau_{B(0,2s_1)}]\Big) + 1 \right]
\end{align*}
and
\begin{align*}
h_2(s_1) = C_3\left(\frac{s_1}{16}\right) \Bigg[C_3\left(L,s_1\right) C_{17}(s_1) +
\ex^0 [\tau_{B(0,2s_1)}] \sup_{|y| \geq \frac{s_1}{4}} \nu(y)\Bigg] + \sup_{|y| \geq \frac{s_1}{16}} \nu(y),
\end{align*}
where
$$
C_{17}(s_1) := K_3(s_1) + \frac{\ex^0 [\tau_{B(0,2s_1)}]}{\left|B(0,\frac{s_1}{4})\right|}
K_2\left(\frac{s_1}{4}, \frac{s_1}{2}, s_1\right)^2.
$$
By (A1) there exists an absolute constant $c_1>0$ such that
$$
K_2(s_1, s_2,\infty) \leq c_1 \quad \text{and} \quad K_2\left(\frac{s_1}{4}, \frac{s_1}{2}, s_1\right) \leq c_1,
$$
and, from \eqref{eq:PruitH} and (A1) we get
$$
\sup_{|y| \geq \frac{s_1}{4}} \nu(y) \leq \sup_{|y| \geq \frac{s_1}{16}} \nu(y \leq c_2 \frac{\Psi(1/s_1)}{s_1^d}.
$$
This, together with \eqref{def:C9}, \eqref{eq:gen_est_tau}, (A3), and the doubling property of the function $\Psi$
yields
\begin{align}\label{eq:h_functions}
h_1(s_1,s_2) \leq c_3 \quad \text{and} \quad h_2(s_1) \leq c_4 \frac{\Psi(1/s_1)}{s_1^d},
\end{align}
with constants $c_3, c_4 >0$. Let $s_1 \geq 1$ arbitrary, and consider a non-negative function $f$ on $\R^d$ which
is regular $(X,V)$-harmonic in a ball $B\big(x,s_1\big)$, $|x| \geq r_0+s_1$. Then, by applying the argument in the
second part of the proof of \cite[Lem. 3.1]{KL17} with the deterministic multiplicative functional $e^{-\eta t}$
replaced by $e^{-\int_0^t V(X_s)ds}$, we get
\begin{align*}
f(y) & \leq \left(\ex^y\left[\int_0^{\tau_{B(x,\frac{s_1}{16})}} e^{-V_{*}(x) t}dt \right] \wedge
\ex^y[\tau_{B(x,\frac{s_1}{16})}]\right) \\
& \ \ \ \ \ \times  \left(h_1(s_1,s_2) \int \limits_{|z-x| >
s_2}f(z) \nu(z-x)dz + h_2(s_1) \int \limits_{\frac{s_1}{8} < |x-z| \leq s_2} f(z)dz \right),
\end{align*}
for all $|y-x| < s_1/32$ and $|x| \geq r_0+s_1$. Finally, by taking $s_1:=\eta|x|$ (then $s_2=2\eta|x|$) with
$|x| \geq r_0/(1-\eta)$ (so that $|x| \geq r_0+\eta|x|$), using \eqref{eq:h_functions} and the estimates
$$
\ex^y\left[\int_0^{\tau_{B(x,\eta|x|/16)}} e^{-V_{*}(x) t}dt \right] \leq \frac{c_5}{V_{*}(x)} \quad \mbox{and}
\quad \ex^y[\tau_{B(x,\eta|x|/16)}] \leq \frac{c_6}{\Psi(1/|x|)}
$$
(recall that $V_{*}(x) \geq 0$ and here we use the convention $1/0 = +\infty$), we obtain the claimed bound
$$
f(y)\leq \frac{c_7}{V_{*}(x) \vee \Psi\big(\frac{1}{|x|}\big)} \left(\int \limits_{|z-x| >
2\eta|x|}f(z) \nu(z-x)dz +
\frac{\Psi\big(\frac{1}{|x|}\big)}{|x|^d}\int \limits_{\frac{\eta|x|}{8} < |x-z| \leq 2\eta|x|} f(z)dz \right),
$$
for $|x-y|< (\eta|x|)/32$ and $|x| \geq r_0/(1-\eta)$. This completes the proof.
\end{proof}
We single out two choices of $(X,V)$-harmonic functions of special interest below, for which the above estimate
directly applies.
\begin{corollary}
\label{cor:harm_est}
Let $(X_t)_{t \geq 0}$ be a  L\'evy process with L\'evy-Khintchin exponent $\psi$ as in \eqref{eq:Lchexp} such that
Assumptions (A1)-(A3) and (A4) hold; specifically, let (A4) hold with some $r_0>0$. We have the following.
\begin{itemize}
\item[(1)] If
\begin{align} \label{eq:modelhf}
f(x) = \left\{
\begin{array}{lrl}
\ex^x\left[e^{- \int_0^{\tau_{{\overline{B} (0,r_0)}^c}} V(X_s) ds} \right] & \mbox{for} & x \notin \overline B(0,r_0), \\
1  & \mbox{for} & x \in  \overline B(0,r_0),
\end{array}\right.
\end{align}
then for every $|x| \geq 2r_0$ we have
\begin{align*}
%\label{eq:modelharmest}
f(x) \leq
\frac{C_4}{V_{*}(x) \vee \Psi\big(\frac{1}{|x|}\big)}\left(\int \limits_{|z-x| > \frac{|x|}{2}}f(z) \nu(z-x)dz
+ \frac{\Psi\big(\frac{1}{|x|}\big)}{|x|^d}\int \limits_{\frac{|x|}{32} < |x-z| \leq \frac{|x|}{2}} f(z)dz \right).
\end{align*}
\item[(2)] Let $\eta \in (0,1/4]$. If
\begin{align} \label{eq:modelhf_2}
f(y) = \left\{
\begin{array}{lrl}
\ex^y\left[e^{- \int_0^{\tau_{B(x,\eta|x|)^c}} V(X_s) ds} \varphi(X_{\tau_{B(x,\eta|x|)}^c})\right] &
\mbox{  for  } & y \in B(x, \eta|x|), \\
\varphi(y)  & \mbox{  for  } & y \notin B(x,\eta|x|),
\end{array}\right.
\end{align}
for some $x \notin B(0, r_0/(1-\eta))$ and a function $\varphi:\R^d \to [0,\infty)$, not identically zero, then
for every $|y-x| < \frac{\eta|x|}{32}$ we have
\begin{align*}
%\label{eq:modelharmest}
f(y) \leq \frac{C_4}{V_{*}(x) \vee \Psi\big(\frac{1}{|x|}\big)} \left(\int \limits_{|z-x| > 2\eta|x|}f(z) \nu(z-x)dz
+ \frac{\Psi\big(\frac{1}{|x|}\big)}{|x|^d}\int \limits_{\frac{\eta|x|}{8} < |x-z| \leq 2\eta|x|} f(z)dz \right).
\end{align*}
\end{itemize}
\end{corollary}
\noindent
The above corollary is a straightforward consequence of Lemma \ref{lm:bhi}. Indeed, by the strong Markov property of
the underlying L\'evy process, the function $f$ defined by \eqref{eq:modelhf} is regular $(X,V)$-harmonic in every ball
$B\big(x,|x|/4\big)$, $|x| \geq 2r_0$. Similarly, for given $x \notin \overline B(0, r_0/(1-\eta))$, the function
\eqref{eq:modelhf_2} is regular $(X,V)$-harmonic in a ball $B\big(x,\eta |x|\big)$.

We can now make use of the above estimates and the technical results obtained in the previous section to derive upper
bounds for the zero-energy solutions for potentials satisfying (A4).

\begin{theorem}
\label{thm:upper_bound_pos}
Let $(X_t)_{t \geq 0}$ be a  L\'evy process with L\'evy-Khintchin exponent $\psi$ as in \eqref{eq:Lchexp} such that
Assumptions (A1)-(A3) and (A4) hold; specifically, let (A4) hold with some $r_0>0$. Moreover, let $\varphi$ be a solution
of (\ref{eveq}). Then the following hold.
\begin{itemize}
\item[(1)]
If
$$
\lim_{|x| \to \infty} \frac{\Psi\left(\frac{1}{|x|}\right)}{V_{*}(x)} = 0 \quad \mbox{and} \quad
\int_{|x| > 2r_0} \frac{\nu(x)}{V_{*}(x)} dx < \infty,
$$
then there exists $C>0$ and $R \geq 2r_0$ such that
$$
|\varphi(x)| \leq C \left\|\varphi \right\|_{\infty} \frac{\nu(x)}{V_{*}(x)}, \quad |x| > R.
$$
In particular, $\varphi \in L^1(\R^d)$.

\item[(2)]
If
$$
\lim_{|x| \to \infty} \frac{\Psi\left(\frac{1}{|x|}\right)}{V_{*}(x)} = 0 \quad \mbox{and} \quad
\int_{|x| > 2r_0} \frac{\nu(x)}{V_{*}(x)} dx = \infty,
$$
and
$$
\sup_{|x| \geq 2r_0} \left[\frac{\nu(x)}{V_{*}(x)}
\exp\left(\eta_{*} \int_{2r_0 \leq |z| \leq |x|} \frac{\nu(z)}{V_{*}(z)} dz\right)\right] <  \infty,
$$
with $\eta_{*} :=C^{-1}_1C_4(C_5 \vee 1)C_6$, then there exist $C>0$ and $R \geq 2r_0$ such that
$$
|\varphi(x)| \leq C \left\|\varphi \right\|_{\infty} \frac{\nu(x)}{V^{*}(x)}
\exp\left(\eta_{*} \int_{r_0 \leq |z| \leq |x|} \frac{\nu(z)}{V_{*}(z)} dz\right), \quad |x| > R.
$$
\item[(3)]
If $\liminf_{|x| \to \infty} \frac{\Psi\left(\frac{1}{|x|}\right)}{V_{*}(x)} > 0$, $\varphi \geq 0$,
and $\varphi \in L^p(\R^d)$ for some $p \in (1,\infty)$, then there exists $C>0$ such that
$$
\varphi(x) \leq C \left\|\varphi\right\|_{p} \left(\frac{\left(\int \limits_{|z| > |x|}
\nu(z)^{\frac{p}{p-1}}dz\right)^{\frac{p-1}{p}}}{\Psi\big(\frac{1}{|x|}\big)} + \frac{1}{|x|^{d/p}} \right), \quad |x| > 2r_0.
$$
\end{itemize}
\end{theorem}

\begin{proof}
By applying the resolvent formula \eqref{eq:eig1} with any $\theta >0$ and $D=\overline B(0,r_0)^c$, and then on letting
$\theta \downarrow 0$, we obtain
\begin{align}
\label{eq:aux_harm}
\varphi(x) = \ex^x \left[\tau_{\overline{B}(0,r_0)^c} < \infty; e^{-\int_0^{\tau_{\overline B(0,r_0)^c}} V(X_s)ds}
\varphi(X_{\tau_{\overline{B}(0,r_0)^c}})\right], \quad |x| > r_0,
\end{align}
and hence
$$
|\varphi(x)| \leq \left\|\varphi\right\|_{\infty} \ex^x \left[e^{-\int_0^{\tau_{\overline{B}(0,r_0)^c }} V(X_s)ds} \right],
\quad |x| > r_0.
$$
Next, let $f$ be the function defined in \eqref{eq:modelhf}. With this we obtain
\begin{align}
\label{eq:phibyf}
|\varphi(x)| \leq \left\|\varphi\right\|_{\infty} f(x), \quad |x| > r_0.
\end{align}
By Corollary \ref{cor:harm_est} (1), we get
$$
f(x) \leq \frac{C_4}{V_{*}(x) \vee \Psi\big(\frac{1}{|x|}\big)}
\left(\int \limits_{|z-x| > \frac{|x|}{2}}f(z) \nu(z-x)dz + \frac{\Psi\big(\frac{1}{|x|}\big)}{|x|^d}
\int \limits_{\frac{|x|}{32} < |x-z| \leq \frac{|x|}{2}} f(z)dz \right), \quad |x| > 2r_0.
$$
In particular,
$$
f(x) \leq \frac{C^{-1}_1C_4(C_5 \vee 1)}{v(|x|) \vee w(|x|)}
\left(\int \limits_{|z-x| > \frac{|x|}{2}}f(z) u(|z-x|)dz + \frac{w(|x|)}{|x|^d}\int \limits_{\frac{|x|}{32} < |x-z| \leq
\frac{|x|}{2}} f(z)dz \right), \quad |x| > 2r_0,
$$
where
$$
u(|x|) := C_5 g(|x|), \quad w(|x|):= C_1^{-1}\Psi\left(\frac{1}{|x|}\right), \quad v(|x|) := V_{*}(|x|), \quad |x| \geq 2r_0.
$$

To show parts (1)-(2) of the theorem note that, by assumption, $\lim_{|x| \to \infty} \Psi(1/|x|)/v(|x|) = 0$, thus we may
assume that $r_0$ is large enough such that $\Psi(1/|x|) \vee v(|x|) = v(|x|)$, for $|x| \geq 2r_0$. This means that the
estimate \eqref{eq:1st_iteration} holds with constant $C_{13}=C^{-1}_1C_4(C_5 \vee 1)$, radius $R_0:=2r_0 \vee 1$, and the
above defined functions $u, w$ and $v$. Moreover, assumptions \eqref{eq:doubl_prop}-\eqref{eq:trans} are also satisfied.
Then the first two statements of the theorem follow directly by Lemma \ref{lem:defic1} and \eqref{eq:phibyf}.

Next we prove the remaining part (3). First note that $\varphi \geq 0$ and similarly as in \eqref{eq:aux_harm} we have
$$
\varphi(x) = \ex^x \left[e^{-\int_0^{\tau_{B(x,|x|/4)}  } V(X_s)ds}\varphi(X_{\tau_{B(x,|x|/4)}})\right], \quad |x| > 2r_0.
$$
Corollary \ref{cor:harm_est} (2) gives the bound
$$
\varphi(x) \leq \frac{C_4}{\Psi\big(\frac{1}{|x|}\big)} \left(\int \limits_{|z-x| > \frac{|x|}{2}}\varphi(z) \nu(z-x)dz
+ \frac{\Psi\big(\frac{1}{|x|}\big)}{|x|^d}\int \limits_{\frac{|x|}{32} < |x-z| \leq \frac{|x|}{2}} \varphi(z)dz \right),
\quad |x| > 2r_0,
$$
Finally, by H\"older inequality with suitable $p, q$,
\begin{align*}
\varphi(x)
& \leq
\frac{C_4}{\Psi\big(\frac{1}{|x|}\big)} \left\|\varphi\right\|_p \left(\int \limits_{|z| > \frac{|x|}{2}} \nu(z)^q dz\right)^{1/q}
+ \frac{c}{|x|^d} \left\|\varphi\right\|_p |B(0,|x|/2)|^{1/q} \\
& \leq
c_1 \left\|\varphi\right\|_p \left(\frac{\left(\int \limits_{|z| > |x|} \nu(z)^{\frac{p}{p-1}}dz\right)^{\frac{p-1}{p}}}
{\Psi\big(\frac{1}{|x|}\big)} + \frac{1}{|x|^{d/p}} \right), \quad |x| > 2r_0,
\end{align*}
which completes the proof.
\end{proof}

As it will be seen below, Theorem \ref{thm:upper_bound_pos} (1)-(2) gives sharp upper bounds, provided $\varphi \geq 0$ (compare
with the lower bounds in Theorem \ref{thm:lower_bound_pos}). We will now prove that if $\varphi$ is antisymmetric with respect to
a given $(d-1)$-dimensional hyperplane $\pi$ in $\R^d$ with $\vec{0} \in \pi$, and has a definite sign on both of the corresponding
half-spaces, then the decay rate in (1) of $|\varphi|$ at infinity far away from $\pi$ improves, while the upper bound in (3)
remains unchanged. By rotating the coordinate system if necessary, we may assume that $\pi = \{x \in \R^d:  x_1 = 0\}$. We make
the assumption
\smallskip
\begin{align}
\label{eq:antisymmetry}
  \begin{array}{c}
	 \mbox{$\varphi\big((-x_1, x_2, ..., x_d)\big) = - \varphi\big((x_1,x_2, ..., x_d) \big)$, \ \ \
           $x = (x_1,...,x_d) \in \R^d,$ }
	\vspace{0.2cm} \\
	\mbox{and \ \ $\varphi\big((x_1,...,x_d)\big) \geq 0$ \ \ whenever \ \ $x_1>0$.}
  \end{array}
\end{align}
\smallskip
The next theorem deals with the case when $\varphi$ has no definite sign, but does satisfy \eqref{eq:antisymmetry}.

\begin{theorem}
\label{thm:upper_bound_pos_antisymm}
Let $(X_t)_{t \geq 0}$ be a  L\'evy process with L\'evy-Khintchin exponent $\psi$ as in \eqref{eq:Lchexp} such that
Assumptions (A1)-(A3) and (A4) hold; specifically, let (A4) hold with some $r_0>0$. Suppose that there exist $C_0 > 0$
and $R_0 > 0$ such that
\begin{align}
\label{eq:smoothnu}
|\nu(z_1) - \nu(z_2)| \leq  C_0 \frac{\nu(z_2)}{|z_2|}|z_1-z_2|, \quad |z_1| \geq |z_2| \geq R_0,
\end{align}
Moreover, let $\varphi$ be a solution of \eqref{eveq} such that \eqref{eq:antisymmetry} holds. We have
the following.
\begin{itemize}
\item[(1)]
If $\lim_{|x| \to \infty} \frac{\Psi\left(\frac{1}{|x|}\right)}{V_{*}(x)} = 0$, $\int_{|x| > 2r_0} \frac{\nu(x)}{V_{*}(x)} dx
< \infty$, $\varphi \in L^1(\R^d)$ and there exists a constant $C_1>0$ such that
\begin{align}
\label{eq:doub_monot}
\frac{\nu(y)}{V_{*}(y)} \leq C_1 \frac{\nu(x)}{V_{*}(x)}, \quad |y| \geq |x|/2 \geq 2r_0,
\end{align}
then there exists $C>0$ and $R \geq 2r_0 \vee 4R_0$ such that
$$
|\varphi(x)| \leq C \, (\left\|\varphi \right\|_{\infty} \vee \left\|\varphi \right\|_{1}) \,
\frac{\nu(x)}{V_{*}(x)} \left(\frac{\Psi(1/|x|)}{V_{*}(x)} \vee
\left(\frac{1}{|x|} \int_{2r_0 <|z| < \frac{|x|}{2}} |z_1|\frac{\nu(z)}{V_{*}(z)} dz\right) \right), \quad |x_1| > R.
$$
\item[(2)]
If $\liminf_{|x| \to \infty} \frac{\Psi\left(\frac{1}{|x|}\right)}{V_{*}(x)} > 0$ and $\varphi \in L^p(\R^d)$, $p > 1$,
then there exists $C>0$ such that
$$
|\varphi(x)| \leq C
\left\|\varphi\right\|_{p} \left(\frac{\left(\int \limits_{|z| > |x|} \nu(z)^{\frac{p}{p-1}}dz\right)^{\frac{p-1}{p}}}
{\Psi\big(\frac{1}{|x|}\big)} + \frac{1}{|x|^{d/p}} \right), \quad |x_1| >  2r_0.
$$
\end{itemize}
\end{theorem}

\begin{proof}
Using the resolvent formula \eqref{eq:eig1} with $\theta >0$ and $D=\overline B(x,|x|/4)^c$, and then letting $\theta
\downarrow 0$, we have
$$
\varphi(y) = \ex^y \left[e^{-\int_0^{\tau_{ B(x,|x|/4)}  } V(X_s)ds}\varphi(X_{\tau_{B(x,|x|/4)}})\right], \quad |y-x| <
\frac{|x|}{4}, \ \ |x| \geq 2r_0.
$$
Denote $H_{+}:= \left\{z \in \R^d: z_1 > 0 \right\}$ and $H_{-}:= \left\{z \in \R^d: z_1 < 0 \right\}$. With this notation,
by \eqref{eq:antisymmetry} we obtain
\begin{align}
\varphi(y) & +  \underbrace{\ex^y \left[e^{-\int_0^{\tau_{ B(x,|x|/4)}  } V(X_s)ds}
\varphi_{-}(X_{\tau_{B(x,|x|/4)}}); X_{\tau_{B(x,|x|/4)}} \in H_{-}  \cap B(0,|x|/2)^c \right]}_{=: \, I_1(x,y)} \nonumber \\
& =
\underbrace{\ex^y \left[e^{-\int_0^{\tau_{ B(x,|x|/4)}  } V(X_s)ds}\varphi(X_{\tau_{B(x,|x|/4)}});
X_{\tau_{B(x,|x|/4)}} \in B(0,|x|/2) \right]}_{=: \, I_2(x,y)} \label{eq:antisymm}\\
& \ \ \ \ +
\underbrace{\ex^y \left[e^{-\int_0^{\tau_{ B(x,|x|/4)}  } V(X_s)ds}\varphi_{+}(X_{\tau_{B(x,|x|/4)}}); X_{\tau_{B(x,|x|/4)}}
\in H_{+}  \cap B(0,|x|/2)^c \right]}_{=:\, I_3(x,y)} .\nonumber
\end{align}
In particular,
$$
\varphi(x) = |\varphi(x)| \leq |I_2(x,x)| + I_3(x,x), \quad x_1 \geq 2r_0.
$$
To obtain (1), first we estimate $|I_2(x,x)|$. By the Ikeda-Watanabe formula \eqref{eq:IWF}, the change of variable
$(z_1,...,z_d)=z \mapsto \hat{z}=(-z_1,z_2,...,z_d)$ in the inner integral over $H_{-} \cap B(0,|x|/2)$, and
\eqref{eq:antisymmetry} we have
\begin{align*}
I_2(x,x) & = \int_{B(x,|x|/4)} G^V_D(x,dy) \int_{B(0,|x|/2)} \varphi(z) \nu(y-z) dz \\
       & = \int_{B(x,|x|/4)} G^V_D(x,dy) \int_{H_{+} \cap B(0,|x|/2)} \varphi(z) \left(\nu(y-z) - \nu(y-\hat{z})\right) dz.
\end{align*}
Thus by \eqref{eq:antisymmetry}, (A1), and Theorem \ref{thm:upper_bound_pos} (1) there exist $c_1, c_2, c_3 >0$ and
$\widetilde R \geq 2r_0$ such that for $|x| > 2 \widetilde R \vee 4R_0$ it follows that
\begin{align*}
|I_2(x,x)|
& \leq c_1 \, \frac{\nu(x)}{|x|} \, G^V_D\big(x,B(x,|x|/4)\big) \, \int_{B(0,|x|/2)} |\varphi(z)| |z_1| dz \\
& \leq c_2 \, \frac{\nu(x)}{|x| V_{*}(x)} \, \left(\int_{|z| \leq \widetilde R} |\varphi(z)| |z_1| dz
+ \int_{2r_0 < |z| < \frac{|x|}{2}} \frac{\nu(z)}{V_{*}(z)} |z_1| dz \right) \\
& \leq c_3 \, \frac{\nu(x)}{V_{*}(x)} \frac{1}{|x|} \, \int_{2r_0 < |z| < \frac{|x|}{2}} \frac{\nu(z)}{V_{*}(z)} |z_1| dz.
\end{align*}
To estimate $I_3$, denote
\begin{align}
\label{eq:modelhf_antisymm}
f_x(y) = \left\{
\begin{array}{lrl}
I_3(x,y) & \mbox{ for } & y \in B(x, |x|/4), \\
\varphi(y) \1_{H_{+}  \cap B(0,|x|/2)^c}(y)  & \mbox{ for } & y \notin B(x,|x|/4),
\end{array}\right.
\end{align}
for every $|x| \geq 2r_0$. The function $f_x(y)$ is regular $(X,V)$-harmonic in a ball $B(x,|x|/2)$, and by Corollary
\ref{cor:harm_est} (2) we have
\begin{align*}
f_x(y) \leq \frac{C_4}{V_{*}(x) \vee \Psi\big(\frac{1}{|x|}\big)}
\left(\int \limits_{|z-x| > \frac{|x|}{2}}f_x(z) \nu(z-x)dz + \frac{\Psi\big(\frac{1}{|x|}\big)}{|x|^d}
\int \limits_{\frac{|x|}{32} < |x-z| \leq \frac{|x|}{2}} f_x(z)dz \right)
\end{align*}
as long as $|y-x| < |x|/128$ and $|x| \geq 2r_0$. By the definition in \eqref{eq:modelhf_antisymm} and the
assumption that $\lim_{|x| \to \infty} \frac{\Psi\left(\frac{1}{|x|}\right)}{V_{*}(x)} = 0$, there exists $R \geq 2r_0$
such that the above estimate gives for $|x| \geq R$
\begin{align*}
I_3(x,x) & \leq \frac{C_4}{V_{*}(x)} \int \limits_{|z-x| > \frac{|x|}{2} \atop |z| > \frac{|x|}{2}, \, z_1 > 0}
\varphi(z) \nu(z-x)dz + C_4 \frac{\Psi\big(\frac{1}{|x|}\big)}{|x|^d V_{*}(x)}
\int \limits_{\frac{|x|}{4} \leq |x-z| \leq \frac{|x|}{2}} \varphi(z)dz\\
& \ \ \ \ \ \ \ \ \ \ \ \ \ \ \ \ \ \ \ + \ C_4 \frac{\Psi\big(\frac{1}{|x|}\big)}{|x|^d V_{*}(x)}
\int \limits_{\frac{|x|}{32} < |x-z| < \frac{|x|}{4}} I_3(x,z) dz =: J_1(x) + J_2(x) + J_3(x).
\end{align*}
The terms $J_1$ and $J_2$ can be estimated directly by using Theorem \ref{thm:upper_bound_pos} (1), \eqref{eq:doub_monot}
and \eqref{eq:PruitH}. Indeed, by increasing $R>0$ if necessary, we have
$$
J_1(x) \leq \frac{c_4}{V_{*}(x)} \left(\sup_{|y| < \frac{|x|}{2}} \frac{\nu(y)}{V_{*}(y)} \right)
\int_{|z| > \frac{|x|}{2}} \nu(z) dz
\leq c_5 \frac{\nu(x)}{V_{*}(x)} \frac{\Psi\big(\frac{1}{|x|}\big)}{V_{*}(x)}, \quad |x| \geq R,
$$
and
$$
J_2(x) \leq c_6 \frac{\Psi\big(\frac{1}{|x|}\big)}{V_{*}(x)}\left(\sup_{|y| < \frac{|x|}{2}} \frac{\nu(y)}{V_{*}(y)} \right)
       \leq c_7 \frac{\nu(x)}{V_{*}(x)} \frac{\Psi\big(\frac{1}{|x|}\big)}{V_{*}(x)}, \quad |x| \geq R,
$$
for some positive constants $c_4, ..., c_7$. Moreover, there exists $c_8 >0$ such that
$$
J_3(x) \leq c_8 \frac{\Psi\big(\frac{1}{|x|}\big)}{V_{*}(x)} \sup_{y: \, |y-x| < \frac{|x|}{4}} I_3(x,y), \quad |x| \geq R.
$$
By \eqref{eq:antisymm} and Theorem \ref{thm:upper_bound_pos} (1) we get
\begin{align*}
I_3(x,y)
& \leq
\varphi(y) + \ex^y \left[e^{-\int_0^{\tau_{ B(x,|x|/4)}  } V(X_s)ds}\varphi_{-}(X_{\tau_{B(x,|x|/4)}})\right] \\
& \leq
c_9 \left\|\varphi\right\|_{\infty}\frac{\nu(y)}{V_{*}(y)} + \ex^y \left[e^{-\int_0^{\tau_{ B(x,|x|/4)}  } V(X_s)ds}
\varphi_{-}(X_{\tau_{B(x,|x|/4)}})\right], \quad |y-x| < \frac{|x|}{4}, \ \ x_1 > R,
\end{align*}
and by one more use of \eqref{eq:IWF},
\begin{eqnarray*}
\lefteqn{
\ex^y \left[e^{-\int_0^{\tau_{ B(x,|x|/4)}}V(X_s)ds}\varphi_-(X_{\tau_{B(x,|x|/4)}})\right] } \\
&\qquad =&
c_{10} \int_{B(x,|x|/4)} G^V_{B(x,|x|/4)}(y,z) \int_{\left\{w: \, w_1 <0\right\}} \varphi_-(w) \nu(w-z) dwdz \\
&\qquad \leq &
c_{11} \left\|\varphi\right\|_1 \, \frac{\nu(x)}{V_{*}(x)},
\end{eqnarray*}
for $|y-x| < \frac{|x|}{4}$ and $x_1 > R$. Due to \eqref{eq:doub_monot}, this means that
$$
\sup_{y: \, |y-x| < \frac{|x|}{4}} I_3(x,y) \leq c_{12} \, (\left\|\varphi\right\|_{\infty}
\vee \left\|\varphi\right\|_1) \frac{\nu(x)}{V_{*}(x)}, \quad x_1 > R.
$$
Putting together all the above estimates, we see that the upper bound in assertion (1) holds.

To establish (2), observe that similarly as above we have for $|y-x| < \frac{|x|}{4}$ and $|x| \geq 2r_0$
that
$$
\varphi(y) = \ex^y \left[e^{-\int_0^{\tau_{ B(x,|x|/4)}  } V(X_s)ds}
\left(\varphi_+(X_{\tau_{B(x,|x|/4)}})-\varphi_-(X_{\tau_{B(x,|x|/4)}})\right)\right],
$$
which yields
\begin{align}
\label{eq:aux_est_antisymm}
\varphi(y) \leq \ex^y \left[e^{-\int_0^{\tau_{ B(x,|x|/4)}  } V(X_s)ds}\varphi_+(X_{\tau_{B(x,|x|/4)}})\right] \leq
\varphi(y) + \ex^y \left[\varphi_-(X_{\tau_{B(x,|x|/4)}})\right], \quad |x| \geq 2r_0.
\end{align}
Define
\begin{align*}
f_x(y) = \left\{
\begin{array}{lrl}
\ex^y \left[e^{-\int_0^{\tau_{ B(x,|x|/4)}  } V(X_s)ds}\varphi_+(X_{\tau_{B(x,|x|/4)}})\right] & \mbox{for} & y \in B(x, |x|/4), \\
\varphi_+(y)  & \mbox{for} & y \notin B(x,|x|/4),
\end{array}\right.
\end{align*}
for every $|x| \geq 2r_0$. By the first inequality in \eqref{eq:aux_est_antisymm} and Corollary \ref{cor:harm_est} (2),
we have
\begin{align*}
\varphi(x) \leq\frac{C_4}{\Psi\big(\frac{1}{|x|}\big)} \left(\int \limits_{|z-x| > \frac{|x|}{2}}\varphi_+(z) \nu(z-x)dz
+ \frac{\Psi\big(\frac{1}{|x|}\big)}{|x|^d}\int \limits_{\frac{|x|}{32} < |x-z| \leq \frac{|x|}{2}} f_x(z)dz \right),
\quad x_1 > 2r_0,
\end{align*}
giving
\begin{align}
\label{eq:aux_est_antisymm_2}
\varphi(x) \leq c_{13} \left\|\varphi\right\|_p \frac{\left(\int \limits_{|z| > |x| }
\nu(z)^q dz\right)^{1/q}}{\Psi\big(\frac{1}{|x|}\big)}
+ \frac{c_{14}}{|x|^{d}} \int \limits_{\frac{|x|}{32} < |x-z| \leq \frac{|x|}{2}} f_x(z) dz, \quad x_1 > 2r_0.
\end{align}
It suffices to estimate the latter integral. By the definition of the function $f_x$ and the second inequality in
\eqref{eq:aux_est_antisymm}, we have
$$
\int \limits_{\frac{|x|}{32} < |x-z| \leq \frac{|x|}{2}} f_x(z) dz \leq \int \limits_{\frac{|x|}{32} < |x-z| \leq
\frac{|x|}{2}} |\varphi(z)| dz + \int \limits_{\frac{|x|}{32} < |x-z| \leq
\frac{|x|}{4}} \ex^z \left[\varphi_-(X_{\tau_{B(x,|x|/4)}})\right] dz.
$$
Similarly as above, \eqref{eq:IWF} implies that for all $|z-x|<|x|/4$, $x_1 > 2r_0$, we have
\begin{align*}
\ex^z \left[\varphi_-(X_{\tau_{B(x,|x|/4)}})\right]
& =
c_{15} \int_{B(x,|x|/4)} G_{B(x,|x|/4)}(z,y) \int_{\left\{w: \, w_1 <0\right\}} \varphi_-(w) \nu(w-y) dwdy \\
& \leq
\frac{c_{16}}{\Psi\big(\frac{1}{|x|}\big)} \left\|\varphi\right\|_p \left(\int_{|y| > |x|} \nu(y)^q dy \right)^{1/q},
\end{align*}
and thus
$$
\int \limits_{\frac{|x|}{32} < |x-z| \leq \frac{|x|}{2}} f_x(z) dz
\leq
c_{17} \left\|\varphi\right\|_p |x|^{d/q} + c_{18}\left\|\varphi\right\|_p |x|^d \frac{\left(\int_{|y| > |x|}
\nu(y)^q dy \right)^{1/q}}{\Psi\big(\frac{1}{|x|}\big)}.
$$
Inserting this estimate into \eqref{eq:aux_est_antisymm_2}, the claimed bound in (3) follows.
\end{proof}
\noindent
As it will be seen in specific cases in Section 6 below, by iterating the bounds in (1)-(2), we can often get the
bound with $1/|x|$ instead of $\frac{\Psi(1/|x|)}{V_{*}(x)} \vee \frac{1}{|x|}$.

\subsection{Lower bound}
\noindent
For a given potential $V$ satisfying Assumption (A4) denote
$$
V^{*}(x) := \sup_{|y| \geq \frac{|x|}{2}} V(y), \quad |x| \geq 2 r_0,
$$
and
$$
\Lambda_{B(x,|x|/2)}(x) := \ex^x \left[\int_0^{\tau_{B(x,|x|/2)}} e^{- V^{*}(x) t} dt\right], \quad |x| \geq 2 r_0.
$$
Clearly, $V^{*}(x)$ is a radial and non-increasing function. The auxiliary function $\Lambda_V$ has the suggestive
meaning of lower envelope of the mean lifetime of the L\'evy process under the potential $V$ in a ball $B(x,|x|/2)$,
that is,
\begin{equation}
\label{meanlif}
\Lambda_{B(x,|x|/2)}(x) \leq \ex^x \left[\int_0^{\tau_{B(x,|x|/2)}} e^{- \int_0^t V(X_s) ds} dt\right], \quad |x| \geq 2 r_0.
\end{equation}

The first lemma gives a lower estimate on $\Lambda_{B(x,|x|/2)}(x)$.
\begin{lemma} \label{lem:meanexit}
Let $(X_t)_{t \geq 0}$ be a  L\'evy process with L\'evy-Khintchin exponent $\psi$ as in \eqref{eq:Lchexp} such that
Assumptions (A1)-(A3) hold, and let $V$ be a potential satisfying (A4); specifically let (A4) hold with some $r_0>0$.
Then there exists $C>0$ such that
$$
\Lambda_{B(x,|x|/2)}(x) \geq \frac{C}{V_{*}(x) \vee \Psi\big(\frac{1}{|x|}\big)}, \quad |x| \geq 2r_0.
$$
\end{lemma}

\begin{proof}
First notice that for every $\eta>0$,
$$
\Lambda_{B(x,|x|/2)}(x) \geq \ex^0 \left[\int_0^{\eta} e^{-V^{*}(x) t} dt; \tau_{B(0,|x|/2)} > \eta \right], \quad |x| \geq 2r_0.
$$
Thus
$$
\Lambda_{B(x,|x|/2)}(x) \geq \frac{1-e^{-V^{*}(x) \eta}}{V^{*}(x)} \big(1-\pr^0(\tau_{B(0,|x|/2)} \leq \eta)\big), \quad |x| \geq 2r_0.
$$
Moreover, by \cite[eq. (3.2)]{bib:Pru} combined with \eqref{eq:PruitH}, there exists $c_1 >0$ such that for every $r, \eta >0$
$$
\pr^0(\tau_{B(0,r)} \leq \eta) \leq c_1 \eta \Psi\left(\frac{1}{r}\right),
$$
which gives
$$
1-\pr^0(\tau_{B(0,|x|/2)} \leq \eta) \geq 1 - c_1 \eta \Psi\left(\frac{1}{|x|}\right), \quad |x| \geq 2r_0.
$$
Since the constant $c_1$ is uniform in $\eta > 0$, we may take $\eta:= \frac{1}{2c_1 \Psi\left(\frac{1}{|x|}\right)}$,
which implies that
$$
\Lambda_{B(x,|x|/2)}(x) \geq \frac{1-\exp\left(-\frac{V^{*}(x)}{2c_1 \Psi\left(\frac{1}{|x|}\right)}\right)}{2V^{*}(x)}, \quad |x| \geq 2r_0.
$$
To conclude, it suffices to observe that when $V^{*}(x) \geq \Psi\left(\frac{1}{|x|}\right)$, we have
$$
\Lambda_{B(x,|x|/2)}(x) \geq \frac{1-e^{-\frac{1}{2c_1}}}{2} \frac{1}{V^{*}(x)}, \quad |x| \geq 2r_0,
$$
and when $V^{*}(x) \leq \Psi\left(\frac{1}{|x|}\right)$,
$$
\Lambda_{B(x,|x|/2)}(x) \geq  \frac{1}{2c_2}\frac{1-\exp\left(-\frac{V^{*}(x)}{2c_1 \Psi\left(\frac{1}{|x|}\right)}\right)}
{\frac{V^{*}(x)}{2c_1 \Psi\left(\frac{1}{|x|}\right)}} \frac{1}{\Psi\left(\frac{1}{|x|}\right)}
\geq \frac{e^{\frac{1}{2c_1}}}{2c_2} \frac{1}{\Psi\left(\frac{1}{|x|}\right)}, \quad |x| \geq 2r_0,
$$
as required.
\end{proof}

With this lemma we also have the following estimate.

\begin{lemma} \label{lem:low_aux}
Let $(X_t)_{t \geq 0}$ be a L\'evy process with L\'evy-Khintchin exponent $\psi$ as in \eqref{eq:Lchexp} such that
Assumptions (A1)-(A3) hold and let $V$ be a potential satisfying (A4); specifically let (A4) hold with some $r_0 >0$.
Then for every positive solution $\varphi$ of \eqref{eveq} there exists $C>0$ such that
$$
\varphi(x) \geq \frac{C}{V_{*}(x) \vee \Psi\big(\frac{1}{|x|}\big)} \,
\int_{|z| < |x| \atop |z+x| < |z-x|} \varphi(z) \nu(x-z) dz dy, \quad |x| > 2r_0.
$$
In particular,
$$
\varphi(x) \geq \frac{C}{C_5^2 C_6}\left( \int_{B(0,r_0)} \varphi(z) dz \right)
\frac{\nu(x)}{V_{*}(x) \vee \Psi\big(\frac{1}{|x|}\big)}, \quad |x| > 2r_0.
$$
\end{lemma}

\begin{proof}
By applying the resolvent formula \eqref{eq:eig1} with any $\theta >0$ and $D=B(x,|x|/2)$, and letting $\theta \downarrow 0$,
we get
$$
\varphi(x) \geq \ex^x \left[e^{-V^{*}(x) \tau_{B(x,|x|/2)}} \varphi(X_{\tau_{B(x,|x|/2)}})\right], \quad |x| > 2r_0.
$$
Then by the Ikeda-Watanabe formula \eqref{eq:IWF} and (A1),
\begin{align*}
\ex^x \left[e^{-V^{*}(x) \tau_{B(x,|x|/2)}}\right. & \left.\varphi(X_{\tau_{B(x,|x|/2)}})\right] \\
& \geq \int\limits_{B(x,|x|/2)} \int\limits_0^{\infty} e^{-V^{*}(x) t} p_{B(x,|x|/2)}(t,x,y) dt
\int_{|z| < |x| \atop |z+x| < |z-x|} \varphi(z) \nu(y-z) dz dy \\
& \geq c_2 \, \Lambda_{B(x,|x|/2)}(x) \, \int_{|z| < |x| \atop |z+x| < |z-x|} \varphi(z) \nu(x-z) dz, \quad |x| > 2r_0.
\end{align*}
An application of Lemma \ref{lem:meanexit} completes the proof of the first inequality. The second estimate is a direct
consequence of the first.
\end{proof}

We are now in the position to state the main theorem of this subsection.

\begin{theorem}
\label{thm:lower_bound_pos}
Let $(X_t)_{t \geq 0}$ be a  L\'evy process with L\'evy-Khintchin exponent $\psi$ as in \eqref{eq:Lchexp} such that
Assumptions (A1)-(A3) hold and let $V$ be a potential satisfying (A4); specifically let (A4) hold with some $r_0 >0$.
Let $\varphi$ be a positive solution of \eqref{eveq}. Then the following hold.
\begin{itemize}
\item[(1)]
If $\lim_{|x| \to \infty} \frac{\Psi\left(\frac{1}{|x|}\right)}{V^{*}(x)} = 0$ and $\int_{|x| > 2r_0}
\frac{\nu(x)}{V^{*}(x)} dx < \infty$, then there exist $C>0$ and $R\geq 2r_0$ such that
$$
\varphi(x) \geq C \frac{\nu(x)}{V^{*}(x)}, \quad |x| > R.
$$
\item[(2)]
If $\lim_{|x| \to \infty} \frac{\Psi\left(\frac{1}{|x|}\right)}{V^{*}(x)} = 0$ and $\int_{|x| > 2r_0}
\frac{\nu(x)}{V^{*}(x)} dx = \infty$, then there exist $\eta^{*}, C >0$ and $R \geq 2r_0$ such that
$$
\varphi(x) \geq C \frac{\nu(x)}{V^{*}(x)} \exp\left(\eta^{*}\int_{2r_0 \leq |y| \leq |x|}
\frac{\nu(y)}{V^{*}(y)} dy\right), \quad |x| > R.
$$
\item[(3)]
If $\liminf_{|x| \to \infty} \frac{\Psi\left(\frac{1}{|x|}\right)}{V^{*}(x)} > 0$ and $\int_{|x| > 2r_0}
\frac{\nu(x)}{\Psi\left(\frac{1}{|x|}\right)} dx < \infty$, then there exist $C>0$ and $R\geq 2r_0$ such that
$$
\varphi(x) \geq C \frac{\nu(x)}{\Psi\left(\frac{1}{|x|}\right)}, \quad |x| > R.
$$
\item[(4)]
If $\liminf_{|x| \to \infty} \frac{\Psi\left(\frac{1}{|x|}\right)}{V^{*}(x)} > 0$ and $\int_{|x| > 2r_0}
\frac{\nu(x)}{\Psi\left(\frac{1}{|x|}\right)} dx = \infty$, then there exist $\eta^{*}, C > 0$ and $R \geq 2r_0$
such that
$$
\varphi(x) \geq C \frac{\nu(x)}{\Psi\left(\frac{1}{|x|}\right)} \exp\left(\eta^{*}\int_{2r_0 \leq |y| \leq |x|}
\frac{\nu(y)}{\Psi\left(\frac{1}{|y|}\right)} dy\right), \quad |x| > R.
$$
\end{itemize}
\end{theorem}

\begin{proof}
As noted before, $\varphi \in C_{\rm b}(\R^d)$. First we prove (1)-(2). Let
$$
u(|x|) := \frac{1}{C_5} g(|x|) \ \ \text{and} \ \ v(|x|):= \1_{\left\{1 \leq |x| \leq 2r_0 \right\}}V^{*}(2r_0)
+ \1_{\left\{1 \leq |x| \leq 2r_0 \right\}} V^{*}(|x|), \ \ \ |x| \geq 1.
$$
By our assumptions and Lemma \ref{lem:low_aux}, there exist $c_1>0$ and $R \geq 2r_0$ such that
$$
\varphi(x) \geq c_1 \frac{1}{v(x)} \, \int_{|z| < |x| \atop |z+x| < |z-x|} \varphi(z) u(|x-z|) dz, \quad |x| > R.
$$
Since $0< \varphi \in C_{\rm b}(\R^d)$ and $\int_{|z| < |x| \atop |z+x| < |z-x|} \varphi(z) u(|x-z|) dz \leq c_3
\left\|\varphi\right\|_{\infty}$, for $2r_0  \leq |x| \leq R$, the same inequality holds also for this range of $|x|$.
Thus the assumptions of Lemma \ref{lem:lower_tech} are satisfied with the functions $u$ and $v$. Hence there exist
constants $c_3 >0$ and $\eta^{*} >0$ for which the estimate
$$
\varphi(x) \geq c_3 \frac{\nu(x)}{V^{*}(x)} \exp\left(\eta^{*}\int_{2r_0 \leq |y| \leq |x|} \frac{\nu(y)}{V^{*}(y)}
dy\right), \quad |x| > R,
$$
holds. This implies (1) and (2). The proof of (3) and (4) follows by the same argument as above with the same $u(|x|)$
and $v(|x|):= \frac{1}{C_1} \Psi\big(\frac{1}{|x|}\big)$, $|x| \geq 1$.
\end{proof}

\section{Decay of zero-energy eigenfunctions for potentials negative at infinity}

\noindent
Now we turn to discussing the spatial decay properties of eigenfunctions of non-local Schr\"odinger operators with
\emph{decaying potentials that are negative at infinity}.

\medskip

\begin{itemize}
\item[\textbf{(A5)}]
Let $V \in L^{\infty}(\R^d)$ be such that there exists $r_0 >0$ and $C > 0$ such that
$$
0 \leq -V(x) \leq C \Psi(1/|x|), \quad |x| \geq r_0.
$$
\end{itemize}

\bigskip
\noindent
Notice that under \textbf{(A5)} we have $V(x) \to 0$ as $|x| \to \infty$, i.e., $V$ is indeed a decaying potential.
It also covers potentials with compact support such as potential wells.

We will now prove a counterpart of Theorem \ref{thm:upper_bound_pos_antisymm} (3) in the case when the potential is
negative at infinity and the negative nodal domain of $\varphi$ is a subset of a given half-space. By rotating the
coordinate system if necessary, without loss of generality we can assume that
\smallskip
\begin{align} \label{eq:nodal_neg}
  \begin{array}{c}
	 \mbox{there exists $l \in \R$ such that $\supp \varphi_{-} \subset \left\{y \in \R^d: y_1 < l \right\}$.}
  \end{array}
\end{align}
%\smallskip

\begin{theorem}
\label{thm:upper_bound_neg}
Let $(X_t)_{t \geq 0}$ be a  L\'evy process with L\'evy-Khintchin exponent $\psi$ as in \eqref{eq:Lchexp} such that
Assumptions (A1)-(A3) and (A5) hold; specifically, let (A5) hold with some $r_0>0$. Moreover, suppose that for every
$\varepsilon \in (0,1)$ there exists $M \geq 1$ such that
\begin{align}
\label{eq:Psi_est}
\Psi(r) \leq \varepsilon \Psi(M r), \quad r \in (0,1].
\end{align}
If $\varphi$ is a solution of \eqref{eveq} such that $\varphi \in L^p(\R^d)$, for some $p >1$, and
\eqref{eq:nodal_neg} holds, then for every $\varepsilon \in (0,1)$ there exist $C >0$ and $R > 3r_0$ such that
$$
\varphi(x) \leq C \left(\left\|\varphi\right\|_p \vee \left\|\varphi\right\|_{\infty} \right)
\left(\frac{\left(\int_{|y| >|x|} \nu(y)^{\frac{p}{p-(1+\varepsilon)}} dy\right)^{\frac{p-(1+\varepsilon)}{p}}}
{\Psi(1/|x|)} + \frac{1}{|x|^\frac{d}{p}}\right)^{1-\varepsilon}, \quad x_1 \geq R.
$$
Moreover, if $\varphi(l+x_1,x_2,..., x_d)= -\varphi(l-x_1,x_2,..., x_d)$, $x \in \R^d$, with $l$ given by
\eqref{eq:nodal_neg}, then there exists $\widetilde R > 3r_0 \vee |l|$ such that the same upper bound is true
for $\varphi(x)$ replaced by $|\varphi(x)|$, whenever $|x_1| \geq \widetilde R$.
\end{theorem}

%We see that in real cases of potentials that are negative at infinity (see the examples in the second class for $\nu \leq 1$) we always have that
%\begin{align} \label{pot1}
%\Psi(1/|x|) \asymp |V(x)| \ \text{for large $|x|$ \ \ \ \ \ \ \ or} \ \ \ \ \ \ \ \  \frac{\Psi(1/|x|)}{|V(x)|} \to \infty \ \ \text{as} \ \ |x| \to \infty.
%\end{align}
\begin{proof}
For $\eta \in (0,1/4]$ we denote $D_{\eta}:= B(x, \eta|x|)$, $x \in \R^d$. First note that there exists $\eta \in
(0,1/4]$ such that
\begin{align}
\label{eq:from_Khas}
\sup_{|x| > \frac{r_0 \vee 1}{1-\eta}} \ex^x\left[e^{\int_0^{\tau_{D_{\eta}}} |V(X_t)|  dt} \right] < \infty.
\end{align}
Indeed, by \eqref{eq:gen_est_tau} and Assumption (A5), there is $c>0$ such that
$$
\ex^x\left[\int_0^{\tau_{D_{\eta}}} |V(X_t)|  dt \right] \leq \sup_{y \in D_{\eta}} |V(y)| \, \ex^x [\tau_{D_{\eta}}]
\leq c \frac{\Psi(1/((1-\eta)|x|))}{\Psi(1/(\eta|x|))}, \qquad |x| > \frac{r_0}{1-\eta},
$$
and, by using \eqref{eq:Psi_est}, we can derive from the above that
$$
\sup_{|x| > \frac{r_0 \vee 1}{1-\eta}}\ex^x\left[\int_0^{\tau_{D_{\eta}}} |V(X_t)|  dt \right] < 1,
$$
for sufficiently small $\eta$. Similarly, by the fact that $V \in L^{\infty}(\R^d)$, cf. (A5), and using
\eqref{eq:gen_est_tau},
$$
\sup_{|x| \leq \frac{r_0 \vee 1}{1-\eta}}\ex^x\left[\int_0^{\tau_{D_{\eta}}} |V(X_t)|  dt \right] < 1
$$
for $\eta$ small enough. Hence, by Khasminskii's Lemma we see that there exists $\eta \in (0,\frac{1}{4}]$ such
that \eqref{eq:from_Khas} holds. Next, by applying \eqref{eq:eig1} for $\theta > 0$ we have
$$
\varphi(x) = \theta \ex^x\left[\int_0^{\tau_{D_{\eta}}} e^{\int_0^t \left(|V(X_s)| - \theta \right) ds}
\varphi(X_{t}) dt \right] + \ex^x\left[e^{\int_0^{\tau_{D_{\eta}}} \left(|V(X_s)| - \theta\right) ds}
\varphi(X_{\tau_{D_{\eta}}}) \right], \quad |x| > \frac{r_0 \vee 1}{1-\eta}.
$$
Letting $\theta \to \infty$ this gives
$$
\varphi(x) = \ex^x\left[e^{\int_0^{\tau_{D_{\eta}}} |V(X_s)| ds} \varphi_{+}(X_{\tau_{D_{\eta}}}) \right] -
\ex^x\left[e^{\int_0^{\tau_{D_{\eta}}} |V(X_s)| ds} \varphi_{-}(X_{\tau_{D_{\eta}}}) \right],
\quad |x| > \frac{r_0}{1-\eta},
$$
with $\eta$ specified above, where $\varphi_{\pm}$ denotes the positive and negative parts of $\varphi$,
respectively. (We note that passing to the limit is possible due to \eqref{eq:from_Khas} and dominated
convergence.) In particular,
\begin{align}
\label{pot2}
\varphi(x) \leq \ex^x\left[e^{\int_0^{\tau_{D_{\eta}}} |V(X_s)| ds} \varphi_{+}(X_{\tau_{D_{\eta}}}) \right] ,
\quad |x| > \frac{r_0 \vee 1}{1-\eta},
\end{align}
and
\begin{align}
\label{pot3}
\ex^x\left[ \varphi_{+}(X_{\tau_{D_{\eta}}}) \right] \leq \varphi(x) +
\ex^x\left[e^{\int_0^{\tau_{D_{\eta}}} |V(X_s)| ds} \varphi_{-}(X_{\tau_{D_{\eta}}}) \right],
\quad |x| > \frac{r_0 \vee 1}{1-\eta}.
\end{align}
Furthermore, by H\"older inequality with $\widetilde p, \widetilde q > 1$ such that $1/\widetilde p+1/\widetilde q
=1$, we have
\begin{align}
\label{pot4}
\ex^x\left[e^{\int_0^{\tau_{D_{\eta}}} |V(X_s)| ds} \varphi_{\pm}(X_{\tau_{D_{\eta}}}) \right] \leq
\left(\ex^x\left[e^{\widetilde q \int_0^{\tau_{D_{\eta}}} |V(X_s)| ds}\right]\right)^{1/\widetilde q} \
\left(\ex^x\left[\varphi_{\pm}^{\widetilde p}(X_{\tau_{D_{\eta}}})\right]\right)^{1/\widetilde p},
\quad |x| > \frac{r_0 \vee 1}{1-\eta}.
\end{align}
Again, by Khasminskii's Lemma and by decreasing $\eta$ (dependent on $\widetilde q$) if necessary, we
obtain that
$$
C_{\eta,\widetilde q}:= \sup_{x \in \R^d} \left(\ex^x\left[e^{\widetilde q
\int_0^{\tau_{D_{\eta}}} |V(X_s)| ds}\right]\right)^{1/\widetilde q} < \infty.
$$
Therefore, by \eqref{pot2} and \eqref{pot4}, it follows that
\begin{align}
\label{pot5}
\varphi(x) \leq C_{\eta,\widetilde q}
\left(\ex^x\left[\varphi^{\widetilde p}_{+}(X_{\tau_{D_{\eta}}})\right]\right)^{1/\widetilde p},
\quad |x| > \frac{r_0 \vee 1}{1-\eta},
\end{align}
and
\begin{align} \label{pot6}
\ex^x\left[e^{\int_0^{\tau_{D_{\eta}}} |V(X_s)| ds} \varphi_{-}(X_{\tau_{D_{\eta}}}) \right]
\leq
C_{\eta,\widetilde q}
\left(\ex^x\left[\varphi^{\widetilde p}_{-}(X_{\tau_{D_{\eta}}})\right]\right)^{1/\widetilde p},
\quad |x| > \frac{r_0 \vee 1}{1-\eta}.
\end{align}

Fix now an arbitrarily small $\varepsilon > 0$ and choose $\eta = \eta(\varepsilon) \in (0,\frac{1}{4}]$
small enough such that \eqref{pot5} holds with $\widetilde p= 1+ \varepsilon$ and $C_{\eta, \widetilde q}
< \infty$, where $\widetilde q = 1+ 1/\varepsilon$. In particular,
\begin{align}
\label{pot7}
\varphi(x)^{1+\varepsilon} \leq C_{\eta,1+1/\varepsilon} \,
\ex^x\left[\varphi^{1+\varepsilon}_{+}(X_{\tau_{D_{\eta}}})\right], \quad |x| > \frac{r_0 \vee 1}{1-\eta}.
\end{align}
We define $f_x(y) = \varphi^{1+\varepsilon}_{+}(y)$ for $y \in D^c_{\eta}$, and $f_x(y) =
\ex^y\left[\varphi^{1+\varepsilon}_{+}(X_{\tau_{D_{\eta}}})\right]$ for $y \in D_{\eta}$. Clearly, $f_x$
is $X$-harmonic in $D_{\eta}$. Recall that both $D_{\eta} = B(x,\eta|x|)$ and $f_x$ depend on the position
$x$. By Corollary \ref{cor:harm_est} (2), there exists $c_1=c_1(\eta)$ such that
$$
f_x(x) \leq \frac{c_1}{\Psi(1/|x|)} \int_{|y-x| > 2\eta|x|} f_x(y) \nu(x-y) dy +
\frac{c_1}{|x|^d} \int_{\frac{\eta}{8}|x| < |y-x| < 2\eta|x|} f_x(y) dy, \quad |x| \geq 2(r_0 \vee 1).
$$
Denote the summands at the right hand side above by $I_1(x)$ and $I_2(x)$, respectively. Again, by applying
H\"older inequality with $\widetilde p = p/(1+\varepsilon)$ and $\widetilde q = p /(p-(1+\varepsilon))$, and
using Assumption (A1), it is seen that there exists $c_2= c_2(\varepsilon,p)>0$ such that
\begin{align*}
I_1(x)
& \leq  \frac{c_1}{\Psi(1/|x|)} \left(\int_{\R^d} \varphi^p_{+}(y) dy\right)^{\frac{1+\varepsilon}{p}}
\left(\int_{|y| > 2\eta|x|} \nu(y)^{\frac{p}{p-(1+\varepsilon)}} dy\right)^{\frac{p-(1+\varepsilon)}{p}} \\
& \leq c_2 \left\|\varphi\right\|^{1+\varepsilon}_p \frac{\left(\int_{|y| >|x|}
\nu(y)^{\frac{p}{p-(1+\varepsilon)}} dy\right)^{\frac{p-(1+\varepsilon)}{p}}}{\Psi(1/|x|)},
\qquad |x| \geq 2(r_0 \vee 1) \vee \frac{1}{2\eta}.
\end{align*}
Moreover, by the definition of $f_x$, \eqref{pot3}, and \eqref{pot6} applied with $\widetilde p = p$,
$\widetilde q =q$,
\begin{align*}
I_2(x)
& \leq
\frac{c_1}{|x|^d} \left(\int_{\eta|x| < |y-x| < 2\eta|x|} \left|\varphi(y)\right|^{1+\varepsilon} dy +
\int_{\frac{\eta}{8}|x| < |y-x|
\leq \eta|x|} \ex^y\left[\varphi^{1+\varepsilon}_{+}(X_{\tau_{D_{\eta}}})\right] dy\right)\\
& \leq
\frac{c_1}{|x|^d} \left(\int_{\eta|x| < |y-x| < 2\eta|x|} \left|\varphi(y)\right|^{1+\varepsilon} dy +
\left\|\varphi\right\|_{\infty}^{\varepsilon} \int_{\frac{\eta}{8}|x| < |y-x| \leq \eta|x|} \varphi(y) dy\right. \\
& \ \ \ \ \ \ \ \ \ \left. \ + C_{\eta,q} \, \left\|\varphi\right\|_{\infty}^{\varepsilon}
\int_{\frac{\eta}{8}|x| < |y-x| \leq \eta|x|} \left(\ex^y\left[\varphi^p_{-}(X_{\tau_{D_{\eta}}})\right]\right)^{1/p}
dy\right), \qquad |x| \geq 3(r_0 \vee 1),
\end{align*}
and by a further application of the H\"older inequality with $\widetilde p = p/(1+\varepsilon)$ and $\widetilde q =
p /(p-(1+\varepsilon))$ to the first integral, and with $p, q$ to the second and third, we get
$$
I_2(x)
\leq
\frac{c_3\left\|\varphi\right\|^{1+\varepsilon}_p}{|x|^\frac{(1+\varepsilon)d}{p}}  +
\frac{c_4 \left\|\varphi\right\|_{\infty}^{\varepsilon}}{|x|^\frac{d}{p}}
\left(\left\|\varphi\right\|_p + \left(\int_{|y-x| \leq \eta|x|}
\ex^y\left[\varphi_{-}^p(X_{\tau_{D_{\eta}}})\right] dy\right)^{1/p}\right) , \quad |x| \geq 3(r_0 \vee 1),
$$
with some $c_3 = c_3(\varepsilon, p)$ and $c_4 = c_4(\varepsilon, p)$. It suffices to estimate the latter integral.
By \eqref{eq:IWF}, Assumption (A1), \eqref{eq:nodal_neg} and \eqref{eq:gen_est_tau}, for every $y \in B(x,\eta|x|)$
and $x \in \R^d$ such that $x_1 > 3(r_0 \vee 1)$, we have
\begin{align*}
\ex^y\left[\varphi_{-}^p(X_{\tau_{D_{\eta}}})\right]
& =
c_5 \int_{D_{\eta}} G_{D_{\eta}}(y,z) \int_{D_{\eta}^c} \varphi_{-}^p(w) \nu(w-z)dwdz \\
& =
c_5 \int_{D_{\eta}} G_{D_{\eta}}(y,z) \int_{\left\{w: \, w_1 <l\right\}} \varphi_{-}^p(w) \nu(w-z)dwdz \\
& \leq c_6 \ex^y [\tau_{D_{\eta}}]  \, \left\|\varphi\right\|_p^p \, \nu(x) \\
& \leq c_7 \left\|\varphi\right\|_p^p \frac{\nu(x)}{\Psi\left(\frac{1}{\eta |x|}\right)}
\leq c_8 \left\|\varphi\right\|_p^p \frac{\nu(x)}{\Psi\left(\frac{1}{|x|}\right)},
\end{align*}
with some $c_5, ..., c_8$, possibly depending on $\varepsilon$ via $\eta$. Thus
$$
I_2(x)
\leq
\frac{c_3\left\|\varphi\right\|^{1+\varepsilon}_p}{|x|^\frac{(1+\varepsilon)d}{p}}  +
\frac{c_9 \left\|\varphi\right\|_{\infty}^{\varepsilon}\left\|\varphi\right\|_p}{|x|^\frac{d}{p}}
\left(1 + \left(\frac{\nu(x)|x|^d}{\Psi\left(\frac{1}{|x|}\right)}\right)^{1/p}\right),
\quad x_1 > 3(r_0 \vee 1),
$$
with $c_9=c_9(\varepsilon, p)$. Note also that under (A1) there exists $c_{10}>0$ such that $\nu(x) \leq c_{10}
\Psi\left(\frac{1}{|x|}\right)|x|^{-d}$, $|x| \geq 3(r_0 \vee 1)$. By putting all the above estimates together,
we see that there exists a constant $c_{11}=c_{11}(\varepsilon, p)$ such that
$$
\varphi(x) \leq c_{11} \left(\left\|\varphi\right\|_p \vee \left\|\varphi\right\|_{\infty} \right)
\left(\frac{\left(\int_{|y| >|x|} \nu(y)^{\frac{p}{p-(1+\varepsilon)}} dy\right)^
{\frac{p-(1+\varepsilon)}{p}}}{\Psi(1/|x|)} + \frac{1}{|x|^\frac{d}{p}}\right)^{1/(1+\varepsilon)},
$$
whenever $x_1 > 3(r_0 \vee 1) \vee \frac{1}{2\eta}$, which is the first claimed bound. The second
statement of the theorem follows from this by the antisymmetry argument.
\end{proof}

A further discussion of the potentials negative at infinity in some specific cases will be made at
the end of Section \ref{subsec:fractional} below.

\section{Specific cases and decay mechanisms} %non-existence of bound states}

\subsection{Isotropic and anisotropic fractional Schr\"odinger operators} \label{subsec:fractional}

Let $L^{(\alpha)}$, $\alpha \in (0,2)$, be a family of self-adjoint pseudo-differential operators
determined by their Fourier transforms
$$
\widehat{L^{(\alpha)} f}(\xi) = -\psi^{(\alpha)}(\xi) \widehat{f}(\xi), \quad \xi \in \R^d, \ \
f \in \Dom(L^{(\alpha)}) = \left\{g \in L^2(\R^d): \psi^{(\alpha)} \widehat{g} \in L^2(\R^d) \right\},
$$
where
\begin{align} \label{eq:phi_stable}
\psi^{(\alpha)}(\xi) = \int_{\R^d \setminus \left\{0\right\}} (1-\cos(\xi \cdot z))
\nu^{(\alpha)}(z)dz.
\end{align}
Here we take $\nu^{(\alpha)}(x) = g(x/|x|)|x|^{-d-\alpha}$, $d \geq 1$, where the function
$g: \mathbb S_d \rightarrow (0,\infty)$, with the $d$-dimensional unit sphere $\mathbb S_d$ centered in
the origin, is such that $g(\theta) = g(-\theta)$ and $c_1 \leq g(\theta) \leq c_2$, for every $\theta
\in \mathbb S_d$, with finite positive constants $c_1, c_2$ (cf. \eqref{eq:Lchexp}). Clearly, every
$\nu^{(\alpha)}(z)dz$ is a symmetric L\'evy measure on $\R^d \setminus \left\{0\right\}$ such that
\begin{align}\label{eq:Lm_stable}
\nu^{(\alpha)}(x) \asymp |x|^{-d-\alpha}, \quad x \in \R^d \setminus \left\{0\right\}.
\end{align}
In particular, $\int_{\R^d \setminus \left\{0\right\}} \nu^{(\alpha)}(z)dz = \infty$ and Assumption (A1)
holds. Also, one can easily check that $\psi^{(\alpha)}(\xi) \asymp |\xi|^\alpha$. From this we can easily
see that the maximal function $\Psi$ of the symbol $\psi$ defined in \eqref{eq:Lchexpprof} satisfies
\begin{align} \label{eq:Psi_stable}
\Psi(r) \asymp r^{\alpha}, \quad r >0.
\end{align}
When the spherical density $g$ is non-trivial, the operator $L^{(\alpha)}$ is often called an
\emph{anisotropic fractional Laplacian of order $\alpha/2$}, and the corresponding stochastic process
generated by it is an anisotropic $\alpha$-stable L\'evy process. When $g \equiv C_{d,\alpha}$ for a
constant $C_{d,\alpha} >0$, the operator $L^{(\alpha)} = -(-\Delta)^{\alpha/2}$ is given by the usual
\emph{isotropic fractional Laplacian}, generating a rotationally symmetric L\'evy process.

Note that, by symmetry, for every $t>0$ we have
\begin{align} \label{eq:aux_pt1}
\sup_{x \in \R^d} p(t,x) = p(t,0) = \int_{\R^d} e^{-t \psi(\xi)} d\xi < c_3 t^{-d/\alpha}, \quad t>0,
\end{align}
and, as proven in \cite{GH},
\begin{align} \label{eq:aux_pt2}
p(t,x) \leq c_4 t |x|^{-d-\alpha}, \quad t >0, \ \ x \in \R^d \setminus \left\{0\right\}.
\end{align}
Then \eqref{eq:aux_pt1} gives (A2), and (A3) follows by a combination of \eqref{eq:aux_pt1}, \eqref{eq:aux_pt2}
and \cite[Lem. 2.2]{KL17}.

First we consider potentials that are positive at infinity in the sense of (A4) and look at positive solutions
of \eqref{eq:eign}.

\begin{theorem}
\label{thm:polynomial}
Let $L^{(\alpha)}$, $0 < \alpha < 2$, be a pseudo-differential operator determined by \eqref{eq:phi_stable}
and $V$ be an $X$-Kato class potential for which there exists $r_0 > 0$ such that $V(x) > 0$ and $V(x)
\asymp |x|^{-\beta}$, for $|x| \geq r_0$, with some $\beta>0$. Suppose that there exists a positive function
$\varphi \in C_{\rm b}(\R^d)$ which is a solution of \eqref{eq:eign}. Then the following hold:
\begin{itemize}
\item[(1)]
If $\beta < \alpha$, then there exist constants $C_1, C_2 > 0$ such that
$$
\frac{C_1}{(1+|x|)^{d+\alpha-\beta}} \leq \varphi(x) \leq \frac{C_2}{(1+|x|)^{d+\alpha-\beta}}, \quad x \in \R^d.
$$
In particular, $\varphi \in L^p(\R^d)$, for every $p \geq 1$.
\item[(2)]
If $\beta \geq \alpha$, then there exist $\gamma \in (0,1)$ and a constant $C_3 >0$ such that
$$
\varphi(x) \geq \frac{C_3}{(1+|x|)^{d-\gamma}}, \quad x \in \R^d.
$$
In particular, $\varphi \notin L^p(\R^d)$, for every $p \in [1, \frac{d}{d-\gamma}]$. On the other hand,
if $\varphi \in L^p(\R^d)$ for some $p > 1$, then there exists $C_4 >0$ such that
$$
\varphi(x) \leq \frac{C_4}{(1+|x|)^{d/p}}, \quad x \in \R^d.
$$
In particular, if  $\varphi \in L^p(\R^d)$ with $p = \frac{d}{d-\gamma-\varepsilon}$ for some $\varepsilon
\in (0,1)$, then there exists $C_4 >0$ such that
$$
\varphi(x) \leq \frac{C_4}{(1+|x|)^{d-\gamma-\varepsilon}}, \quad x \in \R^d.
$$
\end{itemize}
\end{theorem}

\begin{proof}
Due to \eqref{eq:Lm_stable} and \eqref{eq:Psi_stable}, the upper bounds in (1)-(2) follow directly from Theorem
\ref{thm:upper_bound_pos} (1) and (3). The corresponding lower estimates are a consequence of Theorem
\ref{thm:lower_bound_pos} (1) and (4), respectively.
\end{proof}

This has the following implication.
\begin{corollary}
Under the assumptions of Theorem \ref{thm:polynomial} we obtain that $\varphi \in L^1(\R^d)$ if and
only if $\alpha > \beta$.
\end{corollary}

\begin{remark}
{\rm
Let $V$ be a potential positive at infinity, and $V(x) \asymp |x|^{-\beta}$ as $|x|\to\infty$, and consider
the fractional Laplacian $(-\Delta)^{\alpha/2}$, $0 < \alpha < 2$. Although the constants are hard to control
in sufficient detail, a calculation using the above estimates shows that if $\beta \geq \alpha$ and
$\frac{C_1 C_{15}}{C_5 C_{10}} \geq d$, then zero is not an eigenvalue of $(-\Delta)^{\alpha/2} + V$. We note
that $C_{10}$ and $C_{15}$ play the more important role here, giving some best constants involving the jump
doubling domination rate and another ratio related to jump activity. Also, from Theorem \ref{thm:polynomial}(2)
we see that whenever $(0,1) \ni \gamma \geq \frac{d}{2}$, which may occur when $d=1$, the operator $H$ has no
zero eigenvalue.
}
\end{remark}

It can already be seen from the above theorem that there is a transition in the localization properties of
$\varphi$ when $\alpha > \beta$ changes to $\alpha \leq \beta$. For a closer understanding of this transition
around $\alpha \approx \beta$, we consider a more refined class of potentials.

\begin{theorem} \label{thm:polynomial+logarithmic}
Let $L^{(\alpha)}$, $0 < \alpha < 2$, be a pseudo-differential operator determined by \eqref{eq:phi_stable}
and $V$ be an $X$-Kato class potential for which there exists $r_0 > 0$ such that $V(x) > 0$ and $V(x)
\asymp |x|^{-\alpha} (\log|x|)^{\delta}$, for $|x| \geq r_0$, with some $\delta>0$. Suppose that there exists
a positive function $\varphi \in C_{\rm b}(\R^d)$ which is a solution of \eqref{eq:eign}. Then the following
hold.
\begin{itemize}
\item[(1)]
If $\delta > 1$, then there exist constants $C_1, C_2 > 0$ such that
$$
\frac{C_1}{(1+|x|)^d \log(1+|x|)^{\delta}} \leq \varphi(x) \leq \frac{C_2}{(1+|x|)^d \log(1+|x|)^{\delta}},
\quad x \in \R^d.
$$
In particular, $\varphi \in L^p(\R^d)$, for every $p \geq 1$.
\item[(2)]
If $\delta = 1$, then there exist $0 < \gamma_1 \leq 1 \leq \gamma_2$ and constants $C_4, C_5 >0$ such that
$$
\frac{C_1}{(1+|x|)^d \log(1+|x|)^{1-\gamma_1}} \leq \varphi(x) \leq \frac{C_2}{(1+|x|)^d \log(1+|x|)^{1-\gamma_2}},
\quad x \in \R^d.
$$
In particular, $\varphi \in L^p(\R^d)$ for every $p > 1$, but $\varphi \notin L^1(\R^d)$.
\item[(3)]
If $\delta \in (0,1)$, then there exist $0 < \gamma_1 \leq 1 \leq \gamma_2$ and constants $C_6, C_7 >0$
such that
$$
C_6 \frac{e^{\frac{\gamma_1}{1-\delta} \log|x|^{1-\delta}}}{(1+|x|)^d \log(1+|x|)^{\delta}} \leq \varphi(x)
\leq C_7 \frac{e^{\frac{\gamma_2}{1-\delta} \log|x|^{1-\delta}}}{(1+|x|)^d \log(1+|x|)^{\delta}},
\quad x \in \R^d.
$$
In particular, $\varphi\in L^p(\R^d)$, for every $p > 1$, but $\varphi \notin L^1(\R^d)$.
\item[(4)]
If $\delta \leq 0$, then we have exactly the same bounds and $L^p$-properties as in (2) of Theorem
\ref{thm:polynomial}.
\end{itemize}
\end{theorem}

\begin{proof}
Similarly as above, the upper bounds in (1)-(4) follow by an application of the estimates in Theorem
\ref{thm:upper_bound_pos} (1)-(3); specifically, both upper estimates (2) and (3) result from assertion
(2) of this theorem. The corresponding lower estimates are consequences of the respective bounds in (1),
(2), and (4) of Theorem  \ref{thm:lower_bound_pos}.
\end{proof}

From the results above it is seen that the possible localization properties of the positive zero-energy
eigenfunctions or zero-resonances for decaying potentials positive at infinity splits naturally into
disjoint regimes representing the following three different scenarios. (For the simplicity of the discussion
here, we assume that $V$ is a potential that is positive at infinity and regular enough so that $V(x)
\asymp V^{*}(x) \asymp V_{*}(x)$ far away from the origin). Let $r_0 > 0$ be large enough, and define
$$
h(r) = \int_{r_0 < |x| < r} \frac{\nu(x)}{V(x)} dx \quad \mbox{and} \quad \widetilde h(r) = \int_{r_0 < |x| < r}
\frac{\nu(x)}{\Psi(1/|x|)} dx, \quad r >r_0.
$$
Clearly, $h$ and $\widetilde h$ are bounded functions on $(r_0,\infty)$ if and only if the ratios $\nu/V$ and
$\nu/\Psi(1/|x|)$, respectively, are integrable at infinity. The following situations occur:
\begin{itemize}
\item
\emph{Scenario (1)}:
If $\lim_{|x| \to \infty} \frac{\Psi(1/|x|)}{V(x)} = 0$ and the ratio $\frac{\nu(x)}{V(x)}$
is integrable at infinity, then
$$
\varphi(x) \asymp \frac{\nu(x)}{V(x)},
$$
for large enough $|x|$. In particular, $\varphi \in L^1(\R^d)$. Clearly, in this case the corresponding function
$h$ is bounded.

\item
\emph{Scenario (2)}:
If $\lim_{|x| \to \infty} \frac{\Psi(1/|x|)}{V(x)} = 0$ and the integrability of $\frac{\nu(x)}
{V(x)}$ at infinity breaks down, we have $h(r) \to \infty$ as $r \to \infty$. Hence (1) is no longer true and the
function $h$ contributes into the behaviour of $\varphi$ at infinity like
$$
C_1 \frac{\nu(x)}{V(x)} e^{\gamma_1 h(|x|)}  \leq \varphi(x) \leq C_2 \frac{\nu(x)}{V(x)} e^{\gamma_2 h(|x|)},
$$
for large enough $x$, with some $0< \gamma_1 \leq 1 \leq \gamma_2$ and $C_1, C_2 >0$. Observe that Scenario (1)
differs from (2) by the boundedness of $h$.

\item
\emph{Scenario (3)}:
If $\liminf_{|x| \to \infty} \frac{\Psi(1/|x|)}{V(x)} > 0$, then the fall-off rate of $\varphi$ at infinity rapidly
decreases so that
$$
\varphi(x) \geq C_1 \frac{\nu(x)}{\Psi(1/|x|)} e^{\gamma \widetilde h(|x|)},
$$
for large $x$, with some $\gamma, C_1 >0$. Clearly, in case $\widetilde h$ is bounded, it does not contribute to
the above lower rate.
\end{itemize}

\smallskip
Now we revisit the examples in \eqref{eq:motiv_ex} in the situation of Theorems
\ref{thm:polynomial}-\ref{thm:polynomial+logarithmic}.

\begin{example}
{\rm
Let $\alpha \in (0,2)$, $d \geq 1$, and $L^{(\alpha)} = -(-\Delta)^{\alpha/2}$. In this case
$$
\nu(x) = \frac{C(d,\alpha)}{|x|^{d+\alpha}} \quad \mbox{and} \quad \psi(x) = \Psi(|x|) = |x|^\alpha.
$$
To have $\varphi_\kappa > 0$, we consider $l = 0$ in \eqref{eq:motiv_ex}, and to have positive potentials at
infinity, we consider $\kappa \in (\frac{d-\alpha}{2},\frac{d+\alpha}{2})$, $\alpha \leq d$, in accordance with
\eqref{pos}. With these choices we have the following:
\begin{itemize}
\item[(1)]
If $\frac{d}{2} <\kappa< \frac{d +\alpha}{2}$, then by \eqref{exdecays} we have $V_{\kappa}(x) \asymp |x|^{-\beta}$
as in Theorem \ref{thm:polynomial} with $\beta = \beta(\kappa) = d+\alpha-2\kappa$. Since $0< \beta < \alpha$ for
this range of $\kappa$, we are in Scenario (1) above and
$$
\varphi_{\kappa}(x) \asymp \frac{1}{(1+|x|)^{d+\alpha-\beta}},
$$
with $d+\alpha-\beta = 2\kappa$ as in Theorem \ref{thm:polynomial} (1). Here we recover the exact asymptotic
behaviour of $\varphi_{\kappa}$ at infinity.

\item[(2)]
If $\kappa =  \frac{d}{2}$, then by \eqref{exdecays} we have $V_{d/2}(x) \asymp |x|^{-\alpha} \log|x|$, and
Theorem \ref{thm:polynomial+logarithmic} gives that there exist $0 < \gamma_1 \leq 1 \leq \gamma_2$ and constants
$C_1, C_2 >0$ such that
$$
\frac{C_1}{(1+|x|)^d \log(1+|x|)^{1-\gamma_1}} \leq \varphi_{d/2}(x) \leq \frac{C_2}{(1+|x|)^d \log(1+|x|)^{1-\gamma_2}}.
$$
Hence we are now in scenario (2) above. In particular, this means that we recover the behaviour $\varphi_{\kappa}(x)
\asymp 1/|x|^d$ as in \eqref{eq:motiv_ex}, with a near miss dependent on how close $\gamma_1 \leq 1 \leq \gamma_2$ are
to each other. Clearly, this is a marginal and the most delicate case, and closing the gap would require a more refined
analysis. Note that this special case coincides with the threshold $\kappa$ for which $\varphi_{\kappa}\notin L^1(\R^d)$.

\item[(3)]
If $0 \vee \frac{d-\alpha}{2} < \kappa <\frac{d}{2}$, then by \eqref{exdecays} we have $V(x) \asymp |x|^{-\alpha}$,
placing this case in Scenario (3). Theorem \ref{thm:polynomial} (2) gives that there exist $C_1 >0$ and $\gamma =
\gamma(\kappa) \in (0,1)$ such that
$$
\varphi_\kappa(x) \geq \frac{C_1}{(1+|x|)^{d-\gamma}}.
$$
Moreover, observe that for every $\kappa$ from this range and $p > \frac{d}{2\kappa}$ we have $\varphi_\kappa
\in L^p(\R^d)$. Then by the upper bound in Theorem \ref{thm:polynomial} (2) we also get
$$
\varphi_\kappa (x) \leq \frac{C_2}{(1+|x|)^{2\kappa -\varepsilon}},
$$
for every small enough $\varepsilon > 0$, with some $C_2>0$.
\end{itemize}
}
\end{example}

When the solution of the eigenvalue equation \eqref{eq:eign} is antisymmetric with respect to a given hyperplane and
has a definite sign in each nodal domain, then at least far away from this nodal plane some of our upper estimates
improve significantly. If the solution $\varphi$ is no longer positive on $\R^d$ but it satisfies \eqref{eq:antisymmetry},
then the upper bounds in Theorems \ref{thm:polynomial} (1)-(2) and Theorem \ref{thm:polynomial+logarithmic} (1)-(3) also
hold with $\varphi$ replaced with $|\varphi|$ (this is a consequence of Theorem \ref{thm:upper_bound_pos} (1)-(2) and
Theorem \ref{thm:upper_bound_pos_antisymm} (2)). However, due to Theorem \ref{thm:upper_bound_pos_antisymm} (1), in this
case the upper bound for $|\varphi|$ improves and the upper estimates in Theorems \ref{thm:polynomial} (1) and Theorem
\ref{thm:polynomial+logarithmic} (1)-(2) upgrade as follows.
\begin{corollary}
The following situations occur:
\begin{itemize}
\item[(1)]
Under the assumptions of Theorem \ref{thm:polynomial}, if $\beta < \alpha$ and $\varphi \in C_{\rm b}(\R^d)$ is a
solution of \eqref{eq:eign} such that \eqref{eq:antisymmetry} holds, then there is $C_1 >0$ such that
$$
|\varphi(x)| \leq \frac{C_1}{(1+|x|)^{d+\alpha-\beta}} \frac{1}{(1+|x|\1_{\left\{|x_1|\geq1\right\}}))^{\alpha-\beta}},
\quad x \in \R^d.
$$
\item[(2)]
Under the assumptions of Theorem \ref{thm:polynomial+logarithmic}, if $\delta > 1$ and $\varphi \in C_{\rm b}(\R^d)$
is a solution of \eqref{eq:eign} such that \eqref{eq:antisymmetry} holds, then there is $C_2 >0$ such that
$$
|\varphi(x)| \leq \frac{C_2}{(1+|x|)^d \log(1+|x|)^{\delta}} \frac{1}{(1+\1_{\left\{|x_1|\geq1\right\}}\log|x|))^{\delta}},
\quad x \in \R^d.
$$
\end{itemize}
\end{corollary}

\begin{example}
{\rm
Comparing these upper bounds with the exact behaviours of $\varphi_{\kappa}$ for $l = 1$ and the potentials
$V_{\kappa}(x) \asymp |x|^{-\beta}$ in \eqref{eq:motiv_ex} with $\beta = d +\alpha + 2 -2\kappa$, $\kappa \in
(\frac{d+2}{2},\frac{d+\alpha+2}{2})$, we see that in this case our result is not as sharp as in the cases above,
however, it is still remarkably close to the exact rates. Indeed, we get here the upper bound $|\varphi_{\kappa}(x)|
\leq \frac{C_1}{(1+|x|)^{4(\kappa - 1) - d}}$, while the true behaviour is $|\varphi_{\kappa}(x)| \asymp
\frac{1}{|x|^{2\kappa-1}}$.
%(note that $4(\kappa - 1) - d \in (d,d+2\alpha) \subset (d,d+1)$ and $2\kappa-1 \in (d+1, d+3)$).
We emphasize that the symmetry/antisymmetry properties of eigenfunctions as in \eqref{eq:antisymmetry} are
of much interest in spectral theory and are known to have important consequences (see, e.g., \cite{bib:Kal}).
}
\end{example}

Finally, we present a result for the case of fractional Schr\"odinger operators with potentials that are
negative at infinity.

\begin{theorem} \label{thm:fractional_negative}
Let $L^{(\alpha)}$, $0 < \alpha < 2$, be a pseudo-differential operator determined by \eqref{eq:phi_stable},
and $V$ be an $X$-Kato class potential for which there exists $r_0 > 0$ such that $V(x) \leq 0$ and $|V(x)|
\leq C |x|^{-\alpha}$, for $|x| \geq r_0$, with some $C>0$. Suppose that there exists a function $\varphi
\in L^p(\R^d)$, for some $p > 1$, which is a solution to \eqref{eq:eign} with the property that there exists
$i \in \left\{1,2,...,d\right\}$ such that $\varphi(x_1,...,-x_i,...,x_d) = - \varphi(x_1,...,x_i,...,x_d)$
and $\supp \varphi_{-} \subset \left\{x \in \R^d: x_i \leq 0\right\}$. Then for every $q \in (0, \frac{d}{p})$
there exist $C=C(q)$ and $R >0$ such that
$$
|\varphi(x)| \leq C \left(\left\|\varphi\right\|_p \vee \left\|\varphi\right\|_{\infty} \right)
\frac{1}{|x|^{q}}, \quad |x_i| \geq R.
$$
\end{theorem}

\begin{proof}
This is a direct application of Theorem \ref{thm:upper_bound_neg} as now we have $\nu(x) \asymp |x|^{-d-\alpha}$
and $\Psi(r) \asymp r^{\alpha}$.
\end{proof}
\noindent
Roughly speaking, the above result says that for potentials that are negative in a neighbourhood of infinity and
for antisymmetric solutions with nodal domains in the corresponding hyperplanes, $L^p$-integrability, with $p >1$,
always gives a polynomial decay of order near to $d/p$, far from the antisymmetry axis. This result can be compared
with the examples in \eqref{eq:motiv_ex}. Specifically, for every $i = 1,2,...,d$ and $l=1$ we can take $P_i(x) =
c x_i$ and
$$
\varphi_{\kappa,i}(x) = \frac{c x_i}{(1+|x|^2)^{\kappa}},
$$
with $\kappa \in \big(1, \frac{d-\alpha}{2}+1\big]$. As seen in \eqref{neg}, this leads to the case of potentials
$V_{\kappa,\alpha}$ negative in a neighbourhood of infinity. We clearly have $|\varphi_{\kappa,i}(x)| \leq C
(1+|x|)^{-d/p}$ and $\varphi_{\kappa,i} \in L^p(\R^d)$, for every $p > \frac{d}{2\kappa-1}$.

\subsection{Layered-type Schr\"odinger operators}

Let $L^{(\alpha,\gamma)}$, $\alpha \in (0,2)$, $\gamma > 2$, be a family of self-adjoint pseudo-differential
operators determined by
$$
\widehat{L^{(\alpha,\gamma)} f}(\xi) = -\psi^{(\alpha,\gamma)}(\xi) \widehat{f}(\xi), \quad \xi \in \R^d, \ \
f \in \Dom(L^{(\alpha)}) = \left\{g \in L^2(\R^d): \psi^{(\alpha,\gamma)} \widehat{g} \in L^2(\R^d) \right\},
$$
with
\begin{align} \label{eq:phi_layered}
\psi^{(\alpha,\gamma)}(\xi) = \int_{\R^d \setminus \left\{0\right\}} (1-\cos(\xi \cdot z))
\nu^{(\alpha,\gamma)}(z)dz,
\end{align}
where $\nu^{(\alpha,\gamma)}(x) = g(x/|x|)|x|^{-d-\alpha}(1 \vee |x|)^{-(\gamma-\alpha)}$, $d \geq 1$.
Here $g: \mathbb S_d \rightarrow (0,\infty)$
is such that $g(\theta) = g(-\theta)$ and $c_1 \leq g(\theta) \leq c_2$, for every $\theta
\in \mathbb S_d$, with finite positive constants $c_1, c_2$. As before, every
$\nu^{(\alpha,\gamma)}(z)dz$ is a symmetric L\'evy measure on $\R^d \setminus \left\{0\right\}$ such that
$$
\nu^{(\alpha,\gamma)}(x) \asymp |x|^{-d-\alpha}(1 \vee |x|)^{-(\gamma-\alpha)}, \quad x \in \R^d \setminus \left\{0\right\}.
$$
In particular, $\int_{\R^d \setminus \left\{0\right\}} \nu^{(\alpha,\gamma)}(z)dz = \infty$,
$\int_{\R^d \setminus \left\{0\right\}} |z|^2\nu^{(\alpha,\gamma)}(z)dz < \infty$, and
$$
\Psi(r) \asymp r^2, \quad r \in (0,1).
$$
Moreover, it follows from \cite{S10} that the probability transition densities $p(t,x)$ exist and satisfy
$$
p(t,x) \leq c_3 \left( (t^{-d/\alpha} \vee t^{-d/2})  \wedge t |x|^{-d-\alpha} (1 \vee |x|)^{-(\gamma-\alpha)}\right),
\quad t >0, \ \ x \in \R^d.
$$
This, in particular, gives (A2) and, together with \cite[Lem. 2.2]{KL17}, implies that Assumption (A3) holds as well.
The operators $L^{(\alpha,\gamma)}$ generate the class of \emph{layered $\alpha$-stable processes}.

We get the following result for potentials that are positive at infinity.

\begin{theorem}
\label{thm:layered}
Let $L^{(\alpha,\gamma)}$, $0 < \alpha < 2$, $\gamma >2$, be a pseudo-differential operator determined by \eqref{eq:phi_layered}
and $V$ be an $X$-Kato class potential for which there exists $r_0 > 0$ such that $V(x) > 0$ and $V(x)
\asymp |x|^{-\beta}$, for $|x| \geq r_0$, with some $\beta>0$. Suppose that there exists a positive function
$\varphi \in C_{\rm b}(\R^d)$ which is a solution of \eqref{eq:eign}. Then the following hold:
\begin{itemize}
\item[(1)]
If $\beta < 2$, then there exist constants $C_1, C_2 > 0$ such that
$$
\frac{C_1}{(1+|x|)^{d+\gamma-\beta}} \leq \varphi(x) \leq \frac{C_2}{(1+|x|)^{d+\gamma-\beta}}, \quad x \in \R^d.
$$
In particular, $\varphi \in L^p(\R^d)$, for every $p \geq 1$.
\item[(2)]
If $\beta \geq 2$, then there exists a constant $C_3 >0$ such that
$$
\varphi(x) \geq \frac{C_3}{(1+|x|)^{d+\gamma-2}}, \quad x \in \R^d.
$$
On the other hand, if $\varphi \in L^p(\R^d)$ for some $p > 1$, then there exists $C_4 >0$ such that
$$
\varphi(x) \leq \frac{C_4}{(1+|x|)^{d/p}}, \quad x \in \R^d.
$$
\end{itemize}
\end{theorem}
\noindent
For potentials negative at infinity a result similar to Theorem \ref{thm:fractional_negative} holds as well.

\subsection{Decay mechanisms}
From the above it is seen that the decay of ground states at zero eigenvalue depends essentially on two factors.
On the one hand, the sign of the potential at infinity makes a qualitative difference, and as seen in the case of
classical Schr\"odinger operators, it has an impact even on the existence of ground states. From the decay results
above one can appreciate that a positive tail of the potential has a (soft) bouncing effect tending to contain
paths in compact regions, while a negative potential leaves more room for the paths to spread out to infinity. This
difference makes the analysis of potentials negative at infinity much more difficult than of potentials positive at
infinity.
%Indeed Lemma ... cannot be extended ... because ...

On the other hand, the decay depends on some mean times spent in some regions by the paths. Using \eqref{eq:gen_est_tau}
and \eqref{meanlif}, we can give another interpretation of the results above, further highlighting the mechanisms.
Assume, for simplicity, that $V(x) \asymp V^{*}(x) \asymp V_{*}(x)$, for large enough $|x|$. Then we see that the
conditions involving the ratios
$$
\frac{\Psi\left(\frac{1}{|x|}\right)}{V(x)} \asymp
\frac{\ex^x \left[\int_0^{\tau_{B(x,|x|/2)}} e^{- \int_0^t V(X_s) ds} dt\right]}{\ex^0[\tau_{B(0,|x|)}]}
$$
in Theorems \ref{thm:upper_bound_pos}-\ref{thm:lower_bound_pos} actually refer to a balance of the mean survival
times of paths in a ball $B(x,|x|/2)$ under the potential versus in the ball $B(0,|x|)$ free of the potential. Due
to the doubling property, these two times are comparable, and describe specific global lifetimes (note that $B(x,|x|/2)$
can also be replaced by $B(x,c|x|)$, $0 < c < 1$, without qualitatively changing the results). This is in sharp contrast
with the case of confining potentials or decaying potentials leading to a strictly negative and sufficiently
low-lying ground state eigenvalue, where the decay is governed by local lifetimes as given in \eqref{meanntime}. When
$$
\ex^x \left[\int_0^{\tau_{B(x,|x|/2)}} e^{- \int_0^t V(X_s) ds} dt\right] = o(\ex^0[\tau_{B(0,|x|)}])
$$
as in Scenarios (1)-(2) above, the potential has a relatively pronounced effect, making the paths favour (large)
neighbourhoods of the origin than (large) neighbourhoods of far out points. This is reflected in the decay behaviours
of $\varphi$ by $V$ entering explicitly in Scenario (1) discussed in Section 6.1.  When, however,
$$
\ex^x \left[\int_0^{\tau_{B(x,|x|/2)}} e^{- \int_0^t V(X_s) ds} dt\right] = O(\ex^0[\tau_{B(0,|x|)}])
$$
as in Scenario (3), the effect of the potential is weak also in relative terms, and the two lifetimes evolve on
the same scale, being very near to (though clearly differing from) the situation of free fluctuations and absence
of a ground state.

\medskip
\noindent
\textbf{Acknowledgments:} JL is pleased to thank Erik Skibsted for pointing out some references on the zero-eigenvalue
case for classical Schr\"odinger operators.

\end{document}